\documentclass[smallextended]{svjour3}       
\smartqed  

\usepackage{graphicx}
\usepackage{amsmath}
\usepackage{caption}
\usepackage{subcaption}			
\usepackage{enumerate}
\usepackage[misc]{ifsym}
\usepackage{enumitem}
\usepackage{amssymb}				
\usepackage{xcolor}					
\usepackage{tikz}
\usetikzlibrary{patterns,decorations.pathreplacing}	
\usepackage[hyphens]{url}
\usepackage{fix-cm}

\def\makeheadbox{{%
\relax}}

\renewcommand{\P}{\mathbb{P}}
\newcommand{\E}{\mathbb{E}}				
\newcommand{\Var}{\mathrm{Var}}			
\newcommand{\Cov}{\mathrm{Cov}}			
\newcommand{\R}{\mathbb{R}}
\newcommand{\N}{\mathbb{N}}

\newcommand{\numberthis}{\addtocounter{equation}{1}\tag{\theequation}}
\newcommand{\diff}{\,\mathrm{d}}			
\newcommand{\eps}{\varepsilon}
\newcommand{\lH}{\underline{H}}
\newcommand{\uH}{\overline{H}}
\newcounter{saveeqna}
\newcommand{\alpheqna}{\setcounter{saveeqna}{\value{equation}}%
\setcounter{equation}{0}%
\renewcommand{\theequation}{\mbox{A\arabic{equation}}}}
\newcommand{\reseteqna}{\setcounter{equation}{\value{saveeqna}}%
\renewcommand{\theequation}{\arabic{equation}}}
\newcommand{\numChanges}{\mathrm{numChanges}}

\newcommand{\revision}[1]{{#1}}
\journalname{Mathematical Geoscience}

\usepackage[authoryear]{natbib}
\bibpunct{}{}{}{a}{}{;}
\def\citeayip#1{\citeauthor{#1} (\citeyear{#1})}

\begin{document}

\title{Quantification of fracture roughness by change probabilities and Hurst exponents
}

\author{Tim Gutjahr        \and
        Sina Hale \and
        Karsten Keller \and
        Philipp Blum \and
        Steffen Winter 
}

\institute{T. Gutjahr (\Letter) \and K. Keller \at
University of Luebeck, Institute of Mathematics, Ratzeburger Allee 160, 23562 Luebeck\\
              Tel.: +49 451 3101 6050\\
              \email{gutjahr@math.uni-luebeck.de} 
        \and
           S. Hale \and P. Blum \at
              Karlsruhe Institute of Technology (KIT), Institute of Applied Geosciences, Kaiserstr. 12, 76131 Karlsruhe, Germany
        \and
        S. Winter \at
             Karlsruhe Institute of Technology (KIT), Institute of Stochastics, Englerstr. 2, 76131 Karls\-ruhe, Germany
              }

\date{} 

\maketitle

\begin{abstract}
The objective of the current study is to utilize an innovative method called ‘change probabilities' for describing fracture roughness.
In order to detect and visualize anisotropy of rock joint surfaces, the roughness of one-dimen\-sional profiles taken in different directions is quantified. The central quantifiers, `change probabilities', are based on counting monotone changes in discretizations of a profile. These probabilities, which are usually varying with the scale, can be reinterpreted as scale-dependent Hurst exponents. For a large class of Gaussian stochastic processes change probabilities are shown to be directly related to the classical Hurst exponent, which generalizes a relationship known for fractional Brownian motion. While being related to this classical roughness measure, the proposed method is more generally applicable, increasing therefore the flexibility of modeling and investigating surface profiles. In particular, it allows a quick and efficient visualization and detection of roughness anisotropy and scale dependence of roughness.  

\keywords{fracture roughness, anisotropy, change probability, fractional Brownian motion, Hurst exponent, scale dependence}

\subclass{60G18 \and 60G15}
\end{abstract}

\section*{Declarations}
{\bf Funding}\\
We gratefully acknowledge funding by the KIT center MathSEE, seed project `\emph{Surface roughness and anisotropy of natural rock joints}'. Financial support by the Federal Ministry of Education and Research (03G0871D) as part of the project ResKin (\emph{Reaction kinetics in reservoir rocks}) is also gratefully acknowledged.\\[2mm]
\noindent {\bf Conflicts of interests}\\
The authors declare that they have no conflict of interest.\\[2mm]
\noindent {\bf Availability of data and material}\\
Not applicable.\\[2mm]
\noindent {\bf Code availability}\\
The produced Matlab-Code is freely available, see \citeayip{Gutjahr-matlab-code}.

\section{Introduction}

Fracture roughness is an omnipresent and important property of natural rock joints. It varies depending on rock type (e.g.,\ \cite{morgan_cracking_2013}) or formation mechanism (e.g.,\ \cite{xie_multifractal_1999}; \cite{vogler_comparison_2017}) and essentially controls the mechanical and hydraulic properties of discontinuities (\cite{magsipoc_2d_2020}). Contact areas between two rough fracture surfaces influence fracture strength and deformation under normal or shear loadings (\cite{bandis_fundamentals_1983}; \cite{tsang_dependence_1983}; \cite{power_topography_1997}; \cite{jiang_estimating_2006}; \cite{brodsky_constraints_2016}). The effective hydraulic aperture and related parameters such as permeability and flow, however, are determined by the shape of the void space within the fracture, which is dependent both on the roughness of the individual fracture surfaces and on the degree of mismatch between them (\cite{brown_fluid_1987}; \cite{roux_physical_1993}; \cite{zimmerman_hydraulic_1996}; \cite{boutt_trapping_2006}; \cite{zhao_effects_2014}). Furthermore, a quantitative description of fracture roughness is a prerequisite for creating synthetic input data sets for numerical fluid flow simulations, which ideally match relevant surface characteristics of natural fractures (e.g.,\ \cite{vogler_comparison_2017}; \cite{magsipoc_2d_2020}). Recently, researchers often use synthetic fractures with self-affine properties (e.g.,\ \cite{kottwitz_hydraulic_2020}; \cite{dong_quantitative_2020}; \cite{seybold_flow_2020}; \cite{yu_effects_2020}). For example, \citeayip{liu_three-dimensional_2020} simulate anisotropic flow during shearing by using a double-rough-walled fracture model, while others examine the control of roughness on fracture permeability or breakthrough curves for conservative solutes (\cite{zambrano_analysis_2019}; \cite{dou_multiscale_2019}). Fracture roughness is therefore linked to rock mechanics, hydrogeology, and also to various geoscientific fields of application, for example geothermal energy, reservoir engineering or geologic disposal of radioactive waste (\cite{li_stochastic_2019}; \cite{stigsson_novel_2019}).

Numerous parameters and methods have 
been proposed to characterize fracture roughness qualitatively and quantitatively (\cite{magsipoc_2d_2020}) such as the well-known joint roughness coefficient, JRC (\cite{barton_review_1973}; \cite{barton_shear_1977}) or various statistical roughness parameters such as $Z_2$ (\cite{myers_characterization_1962}) and $R_p$ (\cite{el-soudani_profilometric_1978}). Today, fractal methods receive increasing attention for the quantitative description of fracture roughness and are applied to various fracture types from nanometer scale (e.g.,\ \cite{siman-tov_nanograins_2013}) to kilometer scale, including studies on earthquake rupture traces (\cite{candela2012}) or coastlines (\cite{renard_constant_2013}). The application of fractal methods is based on the general assumption that fracture surfaces can approximately be described as self-affine fractals (e.g.,\ \cite{mandelbrot_self-affine_1985}; \cite{brown_fluid_1987}; \cite{thompson_effect_1991}; \cite{power_euclidean_1991}; \cite{schmittbuhl_field_1993}; \cite{odling_natural_1994}; \cite{lee_structural_1996}).

Fractal theory allows for detailed analysis of fracture surface morphology and enables upscaling of roughness and related fracture properties from experimental to field scale. In this context, the Hurst exponent (or the related parameter fractal dimension) represents a measure of roughness and characterizes its scaling behavior (\cite{odling_natural_1994}; \cite{issa_fractal_2003}). The Hurst exponent can also be used to investigate the roughness anisotropy of shear fractures or exposed fault surfaces.
Generally, different values are obtained depending on the profile orientation with respect to the original direction of movement of the surface. While \citeayip{candela2012} and \citeayip{corradetti_evaluating_2017} determined Hurst exponents for natural fault surfaces by using two distinct profile directions oriented parallel and perpendicular to slip, \citeayip{xie_multifractal_1999} considered different spatial directions and positions on the fracture surface for calculating fractal dimensions.

In many studies, minimal values of the Hurst exponent were found parallel to the slip direction and increasing values were observed with progressing angular deviation, which implies a higher roughness along the slip direction (e.g.,\ \cite{renard_high_2006}; \cite{candela2012}; \cite{corradetti_evaluating_2017}). In contrast, \citeayip{xie_multifractal_1999} and \citeayip{sagy_evolution_2007} observed that the studied fault surfaces were smoother in slip-parallel direction. In general, these contrasting observations with regard to the directionality of surface roughness may be caused by the diversity of roughness measures used and their interpretation in spite of model violations. In addition, \citeayip{sagy_evolution_2007} showed that natural small-slip and large-slip fault surfaces are each characterized by different geometric features and can therefore differ greatly in roughness along the slip-direction.

For determining fractal parameters various methods are applied: Fourier power spectrum (e.g.,\ \cite{sagy_evolution_2007}; \cite{bistacci_fault_2011}; \cite{candela2012}; \cite{renard_constant_2013}; \cite{corradetti_evaluating_2017}; \cite{corradetti_impact_2020}), variogram analysis (e.g.,\ \cite{huang_applicability_1992}; \cite{mcclean_apparent_2002}), structure function (e.g.,\ \cite{poon_surface_1992}; \cite{odling_natural_1994}) or box counting method (e.g.,\ \cite{malinverno_simple_1990}).  Regardless of the methodology used,  it should be noted that a correct interpretation of fractal parameters requires self-affinity, or at least some kind of scale-invariance.

In this study, an innovative roughness measure is proposed which is linked to the order relation of value triples measured along directional fracture profiles. The method is based on the identification of so-called ‘change patterns’ and subsequent determination of their relative frequency. This ‘change probability’ serves as simple and generally scale-dependent roughness measure.  Under the assumption of some kind of self-affinity, the method offers a simple way for determining the Hurst exponent. In the classical setting of fractional Brownian motion, which is widely used in modeling rock profiles (e.g.,\ \cite{brown_fluid_1987}; \cite{huang_applicability_1992}; \cite{odling_natural_1994}; \cite{rock}), the change probability is directly related to the Hurst exponent, hence also to the autocorrelation function (\cite{coeurjolly:hal-00383120}). Even beyond this setting, the method offers the possibility of visualizing, detecting and specifying morphological anisotropy of shear fractures or fault planes.

The term `change probability' has first been used in \citeayip{doi:10.1137/S0040585X97984991} in the given context. The idea, however, is far from being new. When considering increments between values instead of the values themselves, changes turn into 'zero-crossings'. Given a sequence of data points, a so-called 'zero-crossing' occurs if two consecutive data points differ in their sign. Indicators and probabilities of zero-crossings were applied in various fields such as signal analysis (e.g.,\ \cite{ChangPielEssigmann51}, \cite{EwingTaylor69}). Further background on zero-crossings can be found in \citeayip{doi:10.1137/S0040585X97984991}.

It is important to note that change probabilities are invariant under monotone transformations of a stochastic process. This invariance is interesting from a theoretical viewpoint, because it widens the class of mathematical models to which the proposed methods apply. On the other hand, this invariance is useful from a practical point of view. For example, differently calibrated measuring devices may map the same height information of a fracture surface to different values in the resulting data set (e.g.,\ gray level image or surface profile). A device might also be more sensitive to small differences than another one, so that the actual heights are mapped onto a larger range of values in the resulting image. This makes it rather difficult to compare data. However, since different mappings typically do not change the relative order of the values, change probabilities are not affected by such effects.

The presented approach is simple and economical from a computational point of view, as it is essentially based on comparing values at different locations. No heavy computations as in the estimation of the autocorrelation function or the Fourier spectrum are therefore needed. Furthermore, the approach is in line with a relatively new general trend in time series analysis and signal processing aiming to avoid using the exact metric values and concentrate instead on the ordinal structure of the data (e.g.,\ \cite{zbMATH05702310}, \cite{ZaninEtAl2012}, \cite{AmigoKellerUnakafova2014}). 

The present study is structured as follows: Section~\ref{sec:methods} introduces the concept of change probabilities as the theoretical base for quantifying directional fracture roughness. Furthermore, the relation between change probabilities and the Hurst exponent is discussed, firstly in the case of fractional Brownian motion, and secondly for some more general class of Gaussian processes. In Sect.~\ref{sec:algorithm}, a practical implementation of the proposed methods is described. In Sect.~\ref{sec:experiments}, the developed algorithm for roughness analysis is applied to natural fracture surfaces. On this basis, the novel method and its performance is compared to other methods. In Sect.~\ref{sec:discussion}, the results are discussed. Finally, concluding remarks are provided in Sect.~\ref{sec:conclusion}.

\section{Methods} \label{sec:methods}

To provide a first insight into the concept of ‘change probability’, it is assumed for simplicity that a profile curve – obtained as a linear section of some surface – is described by a real function on some interval $[0,T]$. For some positive $\tau$ being small relative to $T$ and some $t$ such that $0 \leq t$ and $t+2\tau \leq T$, consider the consecutive values $x(t)$, $x(t+\tau)$ and $x(t+2\tau)$. There are six possible ways to order these three values, as depicted in Fig.~\ref{fig:ordinalPatterns}. The idea behind the following method is that values on a smooth curve are more likely to be ordered in such a way that the value in the middle is the second largest out of the three, while for values on a rough curve it is more likely that the value in the middle is the smallest or largest. Out of the six possible orders of the three values, there are four for which the middle value is the largest or smallest one (Fig.~\ref{fig:ordinalPatterns}). These four order patterns will be called ’change patterns’ because in each of them there is a change from an increasing to a decreasing direction or vice versa.

\begin{figure}[h]
\newcommand{\drawpattern}[2]{
\foreach \j [count=\i] in #1{
\coordinate (c-\i) at (\i+#2-1,\j);
\fill (c-\i) circle (0.15);
\ifnum \i>1
	\pgfmathtruncatemacro{\k}{\i-1}
	\draw (c-\k) -- (c-\i);
\fi
}
}
\centering
\begin{tikzpicture}[scale = 0.6]
\drawpattern{{1,3,2}}{0}
\drawpattern{{2,3,1}}{3}
\drawpattern{{2,1,3}}{6}
\drawpattern{{3,1,2}}{9}
\drawpattern{{1,2,3}}{12}
\drawpattern{{3,2,1}}{15}
\draw[thick, decoration={brace,mirror,raise=0cm},decorate] (-0.5,0) -- (11.5,0)
node [pos=0.5,anchor=north,yshift=-0.1cm] {change patterns};
\end{tikzpicture}
\caption{All possible orderings of values at three consecutive points on a profile}
\label{fig:ordinalPatterns}
\end{figure}
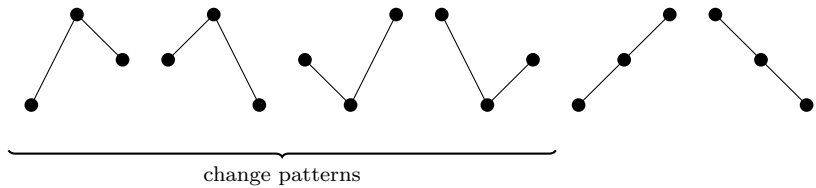

Although the overall aim is the analysis of rock joint surfaces, this section is devoted entirely to the analysis of one-dimensional profiles. Later in Sect.~\ref{sec:algorithm}, the introduced method is used to study the anisotropy of surfaces. Throughout, one-dimensional profiles of rock surfaces are modeled as realizations of a stochastic process $(X_t)_{t\in[0,\infty[}$ with values in ${\mathbb R}$ defined on some probability space with measure $\P$, where the value $X_t$ represents the height of the profile above some reference line.  In other words, it is assumed that such profiles can be represented as graphs of functions, which is debatable but seems to be common practice. It is to some extent justified by the fact that rock surfaces are usually measured using methods such as laser profilometry or white light interferometry which produce such functional data and ignore the true three-dimensional structure of the surface.
It is convenient to assume that the location variable $t\in[0,\infty[$ has a continuous range. In practice, measurements will only be available at finitely many locations $t$ and then it is not a big restriction to set the range to $t\in\{0,1,2,\ldots, n\}$ for some $n\in {\mathbb N}$ as one has the freedom to choose the scale. Details on how to get values $x_0,x_1,\ldots , x_n$ for those $t$ from the rock surface are provided in Sect. \ref{sec:algorithm}.

\subsection{Roughness by change probabilities}

For any stochastic process, a concept of
roughness is provided by so-called 'change probabilities'.
\begin{definition}  Given a stochastic process $(X_t)_{t\in[0,\infty[}$ on some probability space with measure $\P$
and some $\tau \in\ ]0,\infty[$, the {\em change probability of delay} $\tau$ is defined by
\begin{equation*}
p(\tau):= \P(X_0<X_\tau,X_\tau \geq X_{2\tau})+\P(X_0\geq X_\tau,X_{\tau}<X_{2\tau}).
\end{equation*}
\end{definition}
$p(\tau)$ describes how likely it is to observe one of the four change patterns from Fig.~\ref{fig:ordinalPatterns} at location $0$. Although $p(\tau)$ is defined for any process, in order to be a useful quantity for the statistics of the process (in particular to become estimable from change counts), it is reasonable to require the process to have some minimal stationarity, namely to require that
\begin{align}\label{eq:stat}
\P(X_t<X_{t+\tau},X_{t+\tau}\geq X_{t+2\tau})+\P(X_t\geq X_{t+\tau},&X_{t+\tau}<X_{t+2\tau})\notag\\ & \mbox{ does not depend on }t
\end{align}
at least for the locations $t$ and delays $\tau$ of interest.
This \emph{stationarity of pattern occurrences} is a rather weak form of stationarity. It is for instance implied by the common assumption of stationary increments and, in case of a Gaussian process, in particular by the conditions \eqref{a2} and \eqref{a3} imposed later on in this section.
For this and the following, see e.g.~\citeayip{zbMATH05290232}.

If $(X_t)_{t\in[0,\infty[}$ is a Gaussian process with zero mean, then $p(\tau)$ is known to be given by the formula
\begin{align} \label{eq:Gaussian-formula}
p(\tau)=1-\frac{2}{\pi} \arcsin\left(\sqrt{\frac{\Cov(X_\tau-X_0,X_{2\tau}-X_\tau)}{2\sqrt{\Var(X_\tau-X_0)\Var(X_{2\tau}-X_\tau)}}+\frac{1}{2}}\right),
\end{align}
for all $\tau  \in\ ]0,\infty[$.

Change probabilities are proposed here as a measure of roughness. Below it is argued that they are more generally defined and easier to compute than the Hurst exponent that is typically used. In certain situations there is a direct relation between these two roughness measures.
Via their dependence on the delay parameter $\tau$,  addressing different scales,  change probabilities are able to capture more roughness information than the Hurst exponent.

\subsection{Fractional Brownian Motion, change probabilities and the Hurst exponent}\label{sec:fbm}

As already mentioned, fractional Brownian motion (fBm) is widely used as a model for rock profiles, and the main model parameter, the Hurst exponent $H\in\ ]0,1[$, serves as a quantifier of roughness  (see e.g.,~\cite{rock}). Recall that fBm is a Gaussian stochastic process $(X_t)_{t\in [0,\infty [}$ satisfying
\alpheqna
\begin{align}
X_0&=0\quad \P\mbox{-almost surely},\label{a1}\\
\E(X_t) &= 0\quad \mbox{ for all }t\in [0,\infty[,\label{a2}
\end{align}
\reseteqna
and
\begin{equation*}
\Cov(X_s,X_t) =\frac{1}{2}\left (s^{2H}+t^{2H}-\lvert t-s\rvert^{2H}\right )\quad\mbox{ for all }s,t\in [0,\infty[.
\end{equation*}
It has two outstanding properties, the first one being its {\em self-affinity}, meaning that, for any $\alpha>0$, the processes
\begin{equation*}
(X_{\alpha t})_{t\in [0,\infty[}\mbox{ and }(\alpha^H X_{t})_{t\in [0,\infty[} \quad \mbox{ have the same distribution.}
\end{equation*}
This scaling invariance has the consequence that $p(\tau)$ does not depend on $\tau >0$. The second one is that, for any $\tau\in [0,\infty[$,
\alpheqna
\setcounter{equation}{2}
\begin{equation}\label{a3}
\Cov(X_{t+\tau+1}-X_{t+\tau},X_{t+1}-X_t)\quad \mbox{ does not depend on }t\in [0,\infty[.
\end{equation}
Note that the properties \eqref{a2} and \eqref{a3} together mean that the increment process $(X_{t+1}-X_t)_{t\in [0,\infty [}$ of the stochastic process $(X_t)_{t\in [0,\infty[}$ is weakly stationary. In the Gaussian case this implies its stationarity. For fBm, the increment process is also known as {\em fractional Brownian noise}.
\reseteqna

In order to describe the announced relationship between change probabilities and Hurst exponents, consider the function $h:\ ]0,1[\rightarrow ]-\infty,1[$ given by
\begin{equation}\label{eq:h}
h(x) = 1+\log_2(\sin(\pi(1-x)/2))
\end{equation}
for $x\in\ ]0,1[$. The statement characterizing the Hurst exponent in terms of the change probabilities is the following: {\em If $(X_t)_{t\in [0,\infty[}$ is a fBm with Hurst exponent $H\in]0,1[$, then $h(p(\tau)) = H$ for all $\tau \in \ ]0,\infty[$.}
This statement goes essentially back to \citeayip{coeurjolly:hal-00383120}, who formulated it in terms of zero-crossings in the increment process of a fBm.

It is easy to show that for any fBm $p(\tau) \leq 2/3$. Moreover, $h\left (\frac{2}{3}\right )=0$ and $h$ maps $]0,\frac{2}{3}[$ onto $]0,1[$ in a strictly decreasing way. In particular, the higher the change probability is, the lower the Hurst exponent (and thus the rougher the profile). So for the class of fBms the change probability $p(\tau)$, which does not depend on the scale $\tau$, and the Hurst exponent are equivalent quantifications of roughness. In order to treat the more general situation of
scale-dependent roughness, the following definition is introduced.
\begin{definition}  Given a stochastic process $(X_t)_{t\in[0,\infty[}$ on some probability space with measure $\P$
and some $\tau \in\ ]0,\infty[$, the {\em Hurst exponent of delay} $\tau$ is defined by
$H(\tau)=h(p(\tau))$, where $h$ is as in \eqref{eq:h}.
\end{definition}
\revision{Again it is reasonable to assume at least stationarity of pattern occurrences here as introduced in \eqref{eq:stat}.} Basically, the Hurst exponent of delay $\tau$ is nothing more than a reinterpretation of the change probability of delay $\tau$. However, it allows a unified discussion of roughness for a large class of stochastic processes in the established framework and an easier comparison with the usually considered Hurst exponent $H$ (see \eqref{eq:HurstExp}), whenever it is defined.

\subsection{Generalizing results to a wider class of Gaussian processes} \label{sec:asymptoticProperties}

The self-affinity of fBm results in the same change probability $p(\tau)$ (and thus in the same Hurst exponent $H(\tau)$) at all delays $\tau$. This kind of scale-invariance of roughness is a relatively strong property. It is generally not completely compatible with real world surface profile data, but interesting from a modeling viewpoint. In order to gain some more flexibility on the modeling side, Gaussian processes with asymptotically stabilizing roughness will be considered and investigated, for which the roughness (in the sense of change probabilities) is allowed to vary with the scale. The relation to the Hurst exponent, which ultimately is an asymptotic concept capturing the long-range behavior of a process, will be preserved in form of an asymptotic equality. This relation is based on the following fundamental characterization of the Hurst exponent of delay $\tau$ (and thus of the change probability $p(\tau)$) in terms of variance scaling.

\begin{proposition}\label{propchar} 
Let $(X_t)_{t\in [0,\infty [}$ be a Gaussian stochastic process satisfying \eqref{a1} and \eqref{a2}. If $0<\Var(X_\tau)=\Var(X_{2\tau}-X_\tau)$
for some $\tau >0$, then
\begin{align*}
H(\tau)=\frac{1}{2} \log_2\left(\frac{\Var(X_{2\tau})}{\Var(X_\tau)}\right).
\end{align*}
\end{proposition}

\begin{proof}
Recall formula \eqref{eq:Gaussian-formula} for the change probability of a Gaussian process. Applying the transformation $h$ to $p(\tau)$ yields
\begin{align*}
H(\tau)=h(p(\tau))
&= 1+\frac{1}{2}\log_2\left(\frac{\Cov(X_\tau-X_0,X_{2\tau}-X_\tau)}{2\sqrt{\Var(X_\tau-X_0)\Var(X_{2\tau}-X_\tau)}}+\frac{1}{2}\right)\\
&=\frac{1}{2}+\frac{1}{2}\log_2\left(\frac{\Cov(X_\tau-X_0,X_{2\tau}-X_\tau)}{\sqrt{\Var(X_\tau-X_0)\Var(X_{2\tau}-X_\tau)}}+1\right)\\
&= \frac{1}{2}+\frac{1}{2}\log_2\left(\frac{\Cov(X_\tau,X_{2\tau}-X_\tau)}{\Var(X_\tau)}+1\right)\\
&= \frac{1}{2}+\frac{1}{2}\log_2\left(\frac{\Cov(X_\tau,X_{2\tau})-\Cov(X_\tau,X_\tau)}{\Var(X_\tau)}+1\right)\\
&= \frac{1}{2}+\frac{1}{2}\log_2\left(\frac{\Cov(X_\tau,X_{2\tau})}{\Var(X_\tau)}\right).
\end{align*}
By the assumptions, one has
\[ \Var(X_\tau) 
=\Var(X_{2\tau}-X_\tau) = \Var(X_{2\tau})-2\Cov(X_{2\tau},X_\tau)+\Var(X_\tau), \]
which is equivalent to
\begin{equation}
\Cov(X_\tau,X_{2\tau}) = \frac{1}{2}\Var(X_{2\tau}).
\label{eq:CovVar}
\end{equation}
Hence
\begin{align*}
\hspace{1,5cm} H(\tau)&= \frac{1}{2}+\frac{1}{2} \log_2\left(\frac{\Var(X_{2\tau})}{2\Var(X_{\tau})}\right) = \frac{1}{2} \log_2\left(\frac{\Var(X_{2\tau})}{\Var(X_\tau)}\right).\hspace{1.5cm}\qed
\end{align*}
\end{proof}
Note that if the process $(X_t)$ satisfies condition \eqref{a3}, then the assumption $\Var(X_\tau)=\Var(X_{2\tau}-X_\tau)$ in Proposition~\ref{propchar} is satisfied for all $\tau\in\N$.

 The above proposition shows that the change probabilities (and thus the scale-dependent Hurst exponents) of a Gaussian process are, under some stationarity conditions, related to its variance, which in turn is related to the covariance, as can be seen from \eqref{eq:CovVar}.
In the following the classical Hurst exponent will be discussed, which is based on the autocorrelation and therefore as well on the covariance of the process. In particular, a relationship between scale-dependent Hurst exponents and their classical counterpart will be established.

For a stochastic process $(X_t)_{t\in [0,\infty ]}$ satisfying \eqref{a1}, \eqref{a2}, \eqref{a3} and
\alpheqna
\setcounter{equation}{3}
\begin{equation}\label{a4}
0<\Var(X_t)<\infty\quad \mbox{ for all }t\in\ ]0,\infty[,
\end{equation}
\reseteqna
consider the \emph{autocorrelation function} $c:[0,\infty[ \rightarrow \R$ of the increment process $(X_{t+1}-X_t)_{t\in [0,\infty [}$ which is
defined by \[ c(\tau):= \Cov(X_{\tau+1}-X_\tau,X_1)/\Var(X_1), \]
for $\tau \in [0,\infty[$. (In case $\Var(X_1)=0$, not relevant at this point, one defines $c(\tau)=0$.)

Recall that for fBm with Hurst exponent $H$ it holds
\[c(\tau)=\frac{1}{2}\left (\lvert\tau +1\rvert^{2H}-2\lvert \tau\rvert^{2H}+\lvert\tau -1\rvert^{2H} \right ),
\]
for all $\tau \in [0,\infty[$ and
\begin{equation*}
\lim_{\tau\to\infty} \frac{c(\tau)}{\tau^{2H-2}}=H(2H-1),
\end{equation*}
cf.~e.g.\ \citeayip{Gneiting2004}.
The latter formula describes the asymptotic behavior of $c$ for large $\tau$ and gives rise to a general definition of the {\em Hurst exponent} $H$ of $(X_t)_{t\in \N_0}$ as the number
\begin{equation}\label{eq:HurstExp}
H:= 1+\frac{1}{2}\lim_{\tau \to \infty} \frac{\log(\lvert c(\tau) \rvert)}{\log(\tau)},
\end{equation}
provided this limit exists, see e.g.\ \citeayip{Gneiting2004}. To see that this definition is consistent with the role of the Hurst parameter for fBm, assume that the limit
\alpheqna
\setcounter{equation}{4}
\begin{equation}\label{a5}
\lim_{\tau\to\infty} \frac{c(\tau)}{\tau^{2H-2}}\mbox{ exists for some }H\mbox{ and differs from }0.
\end{equation}
\reseteqna
Then $H$ is necessarily the Hurst exponent of $(X_t)_{t\in [0,\infty [}$ as defined by \eqref{eq:HurstExp}, which is easily deduced from \eqref{a5}, see Lemma~\ref{lem:3} in the Appendix.
Moreover, if $H\geq 0.5$, then the limit in \eqref{a5} is necessarily positive, see Lemma~\ref{lem:2} in the Appendix. (The last assumption in \eqref{a5} implies in particular that the quotient ${c(\tau)}/{\tau^{2H-2}}$ is bounded away from zero for large $\tau$: there exists some constants $C,\tau_0>0$  such that either $c(\tau)\tau^{2-2H}>C$ for all $\tau\geq \tau_0$ or $-c(\tau)\tau^{2-2H}>C$ for all $\tau\geq\tau_0$. This mean that $c$ will not change its sign for $\tau\geq \tau_0$.)

In the sequel, all considerations will be restricted to Gaussian processes $(X_t)_{t\in [0,\infty[}$ satisfying the conditions \eqref{a1}-\eqref{a5}. By the above, this implies in particular that the Hurst exponent of $(X_t)$ exists. Note that if $(X_t)$ is not a fBm, then the values of $p(\tau)$ may vary with $\tau\in\ ]0,\infty[$, and therefore $H(\tau)=h(p(\tau))$ can be different from the Hurst exponent for some $\tau\in\ ]0,\infty[$. However, the above conditions ensure that the Hurst exponent $H(\tau)$ of delay $\tau$ converges to the Hurst exponent $H$ as $\tau\to\infty$.

\begin{theorem}\label{thm:1}
Let $(X_t)_{t_\in[0,\infty[}$ be a Gaussian process satisfying the conditions \eqref{a1}-\eqref{a4}. Assume that condition \eqref{a5} is satisfied for some constant $H$ such that $0.5\leq H<1$.
Then the Hurst exponent (as defined in \eqref{eq:HurstExp}) exists and equals $H$ and
\begin{equation*}
\liminf_{\tau \to \infty}\ H(\tau) \leq H \leq \limsup_{\tau \to \infty}\ H(\tau).
\end{equation*}
holds true. Thus, if the limit
\begin{equation}\label{eq:limm}
\lim_{\tau \to \infty}H(\tau)
\end{equation}
exists, then it necessarily equals $H$. 
\end{theorem}
Theorem~\ref{thm:1} establishes an asymptotic relation between change probabilities and the Hurst exponent, which is the best one can hope for in the situation when there is no self-affinity present.  As the assumptions in Theorem~\ref{thm:1} are rather mild, one can expect them to be satisfied by many stochastic processes. Examples covered by Theorem~\ref{thm:1} that are not fBm are for instance provided by the processes discussed in \citeayip{Gneiting2004}: if $(X_t)$ is a Gaussian process satisfying \eqref{a1} such that its increment process is stationary and has a correlation function from the Cauchy class (or any of the other classes discussed in \citeayip{Gneiting2004}), then the conditions \eqref{a2}-\eqref{a5} are satisfied. The proof of Theorem~\ref{thm:1} can be found in the Appendix.

Notice that for data of rock surfaces the Hurst exponent is typically larger than 0.5 (see e.g.,\ \cite{candela2012}). Therefore, the condition $H\geq 0.5$ is not a serious restriction from the practical viewpoint.

\revision{Compared to the classical setting of fBm, the more general class of processes behind Theorem~\ref{thm:1} increases flexibility when working with Hurst exponents as a roughness measure. Now roughness is allowed to vary with the delay. Only some stabilization for increasing delays is assumed, allowing to still consider a single Hurst exponent. This provides a theoretical framework for the case that some stabilization of (estimated) Hurst exponents is observed for real data.}

\subsection{Estimating roughness}

Given a stochastic process $(X_t)_{t\in\ [0,\infty[}$ satisfying \eqref{eq:stat} for $t\in {\mathbb N}_0$, a natural
estimator of $p(\tau)$ for some $\tau\in\ ]0,\infty [$ and $N\in {\mathbb N}$ is given by
\begin{eqnarray*}
\widehat p^{(N)}(\tau) := \frac{1}{N} \sum_{t=0}^N (1_{\{X_t<X_{t+\tau},X_{t+\tau}\geq X_{t+2\tau}\}}+1_{\{X_t\geq X_{t+\tau},X_{t+\tau}< X_{t+2\tau}\}}).
\end{eqnarray*}
The following considerations are restricted to the case $\tau\in {\mathbb N}$. Principally, it is enough to start with a process $(X_t)_{t\in {\mathbb N}_0}$. Assume that for such a process, the increment process $(X_{t+1}-X_t)_{t\in {\mathbb N}_0}$ is a non-degenerated zero-mean stationary Gaussian process. Non-degenerated means that each of the finite distributions of the process is absolutely continuous with respect to the Lebesgue measure.

If $c(\tau)\to 0$ as $\tau\to\infty$, then $\widehat p^{(N)}(\tau)$ is a strongly consistent and asymptotically unbiased estimator  ($N\to \infty$)
of $p(\tau)$ for $\tau\in {\mathbb N}$. If, moreover,
\begin{eqnarray}\label{eq:betacond}
\lvert c(\tau) \rvert\cdot \tau^\beta \to 0\mbox{ as }\tau \to \infty\mbox{ for some }\beta<0.5,
\end{eqnarray}
then the estimator is asymptotically normally distributed as $N\to \infty$. This has been shown in \citeayip{Sinn2011} in the general context of ordinal pattern probabilities for $\tau=1$ using an asymptotic result of \citeayip{zbMATH00828385}; the generalization to $\tau\neq 1$ is straightforward. The idea goes back to the work of \citeayip {HO1987144}, where zero-crossings in a process are discussed. Note that a change in the original process corresponds to a zero-crossing in the associated increment process. Note also that \eqref{eq:betacond} is satisfied if condition \eqref{a5} holds with $0<H<0.75$.

When passing from $\widehat p^{(N)}(\tau)$ to $\widehat{H}^{(N)}(\tau):=h(\widehat p^{(N)}(\tau))$ and from $p(\tau)$ to $H(\tau)=h(p(\tau))$, all described estimation properties remain true, which follows by applying the delta method (see e.g.,\ \cite{deltaMethod}). In particular, for each $\tau\in {\mathbb N}$, $\widehat H^{(N)}(\tau)$ is an estimator of the Hurst exponent $H(\tau)$ with the described properties. It is interesting to note that the estimators $\widehat p^{(N)}(\tau)$ and $\widehat{H}^{(N)}(\tau)$, which are only based on the ordinal structure of a time series, show a good performance, in particular have a low bias (see \cite{Sinn2011}).

In the situation of Theorem \ref{thm:1}, the Hurst exponent $H$ can still be estimated by $\widehat{H}^{(N)}(2^n)$ with sufficiently large $n$ and $N$.


\section{An Algorithm for Visualizing and Detecting Fracture Roughness Anisotropy} \label{sec:algorithm}

In this section, an algorithm is described that takes a two-dimensional gray level image as input and returns an array of estimates of the Hurst exponents $H(\tau)$ for a number of directions $\phi$ and delays $\tau$. The algorithm may be viewed as an application of Theorem~\ref{thm:1}. In this theorem, the data is modeled by a stochastic process $X_t$ with $t \in \R$, while in practice one is dealing with discrete data, i.e.\ $X_t$ with $t\in \N_0$. Here it is necessary to determine what $t=1$ should mean, i.e.\ to fix a minimal spatial scale.
Since pixel images are considered here, it makes sense to set the width of one pixel to 1, which is the smallest meaningful distance between two points in these images.
However, when comparing images of different resolutions, the size of a pixel might correspond to different physical distances and then one has to be careful to compare quantities of physically matching delays.

\subsection{Input}

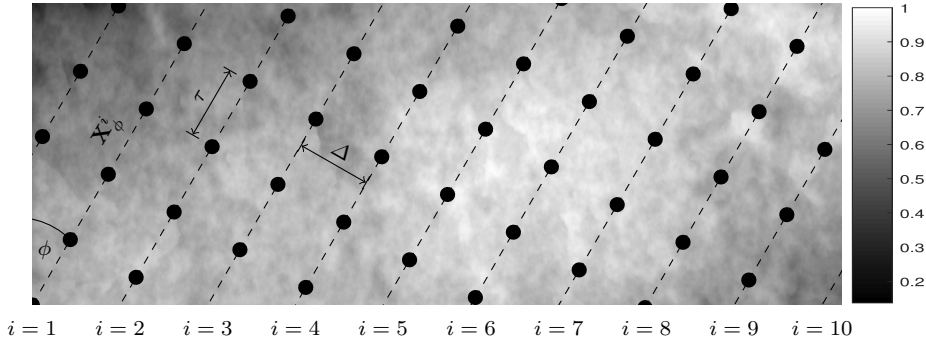
\begin{figure}[ht]
\centering
\begin{tikzpicture}
\pgfmathsetmacro{\a}{10.65};	
\pgfmathsetmacro{\b}{4};	
\pgfmathsetmacro{\angle}{30};	
\pgfmathsetmacro{\dist}{1}; 

\node[inner sep = 0] at (\a/2,\b/2) {\includegraphics[width=\a cm, height=\b cm]{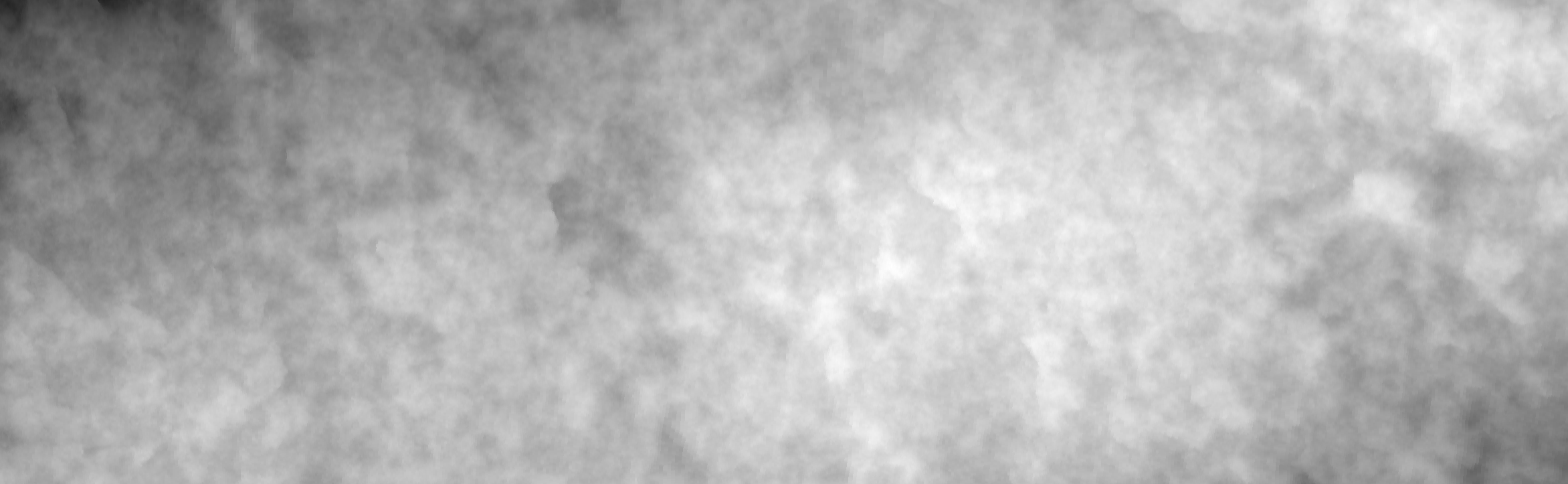}};


\pgfmathtruncatemacro{\m}{ceil(max(\a,\b)/\dist)};
\pgfmathsetmacro{\d}{1/cos(\angle)};
\foreach \i in {0,...,9}{
\pgfmathtruncatemacro{\j}{\i+1}
\node at (\i*\d,-0.3) {{\small $i=\j$}};
}
\begin{scope}
	\clip (0,0) rectangle (\a,\b);
	\begin{scope}[rotate = -\angle]
		\draw (0,1)  arc[radius = 1, start angle= 75, end angle= 180];
		\node at (-0.23,0.72) {$\phi$};
		\node[rotate = 90-\angle] at (-0.3,2.5) {$\mathbf{X}_\phi^i$};
		\foreach \x in {-\m,...,\m}{
			\foreach \y in {-\m,...,\m}{
				\draw[dashed] (\x,\y) -- (\x,\y+1);
				\fill (\x,\y) circle (0.1);
			}
		}
		\draw[|<->|] (0.7,3) -- node[above,rotate = 90-\angle] {$\tau$} (0.7,4);
		\draw[|<->|] (2,3.6) -- node[above,rotate = -\angle] {$\Delta$} (3,3.6);
	\end{scope}
\end{scope}

\node[inner sep = 0] at (\a+0.6,\b/2) {\includegraphics[width=1 cm, height=\b cm]{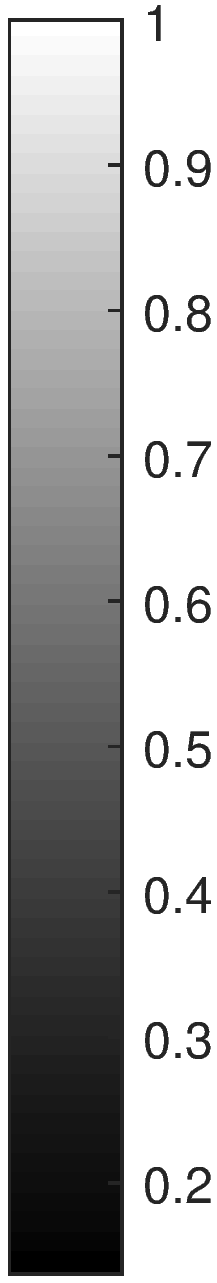}};
\end{tikzpicture}
\caption{
Exemplary input matrix of a natural sandstone fracture surface. The value of one corresponds to the maximum height of the fracture surface, the value zero corresponds to the minimum height. In step~\ref{step3} of the algorithm, parallel one-dimensional profiles $\mathbf{X}_\phi^i$ are extracted for a given angle $\phi$}
\label{fig:v(d,phi)}
\end{figure}

The input of the algorithm is a two-dimensional matrix $A=(a_{i,j})\in \R^{w_x\times w_y}$ of width $w_x\in \N$ and height $w_y\in \N$, where the entry $a_{i,j}$ corresponds to the relative height of the two-dimensional rock surface at the position $(i,j)$. This matrix can be represented as a gray level image, where the brightness of the pixel at position $(i,j)$ is equal to the value of $a_{i,j}$.

Additional inputs are $n_\phi \in \mathbb N$, a parameter of the algorithm fixing the number of angular directions $\phi$ for which profile lines are considered, and a vector $\overrightarrow{\tau}$ of length $n_\tau$ containing the delays $\tau$ at which the one-dimensional profiles will be scanned. The algorithm will choose $n_\phi$ angles equally spaced between $0^{\circ}$ and $180^{\circ}$ and the delays contained in $\overrightarrow{\tau}$ are positive integers. Optional input is a parameter $\Delta$ which fixes the distance between neighboring parallel profile lines (Fig.~\ref{fig:v(d,phi)}). The default value of $\Delta$ is $1$. Recall that value $1$ for the delay $\tau$ or the parameter $\Delta$ corresponds to a length equal to the width of a pixel.

\subsection{Algorithm} \label{subsec:algorithm}

\begin{figure}[ht]
\scalebox{0.65}{
\begin{tikzpicture}
\node[inner sep = 0pt] (Step1) at (-0.5,-0.5) {\includegraphics[width=6cm, height=6cm]{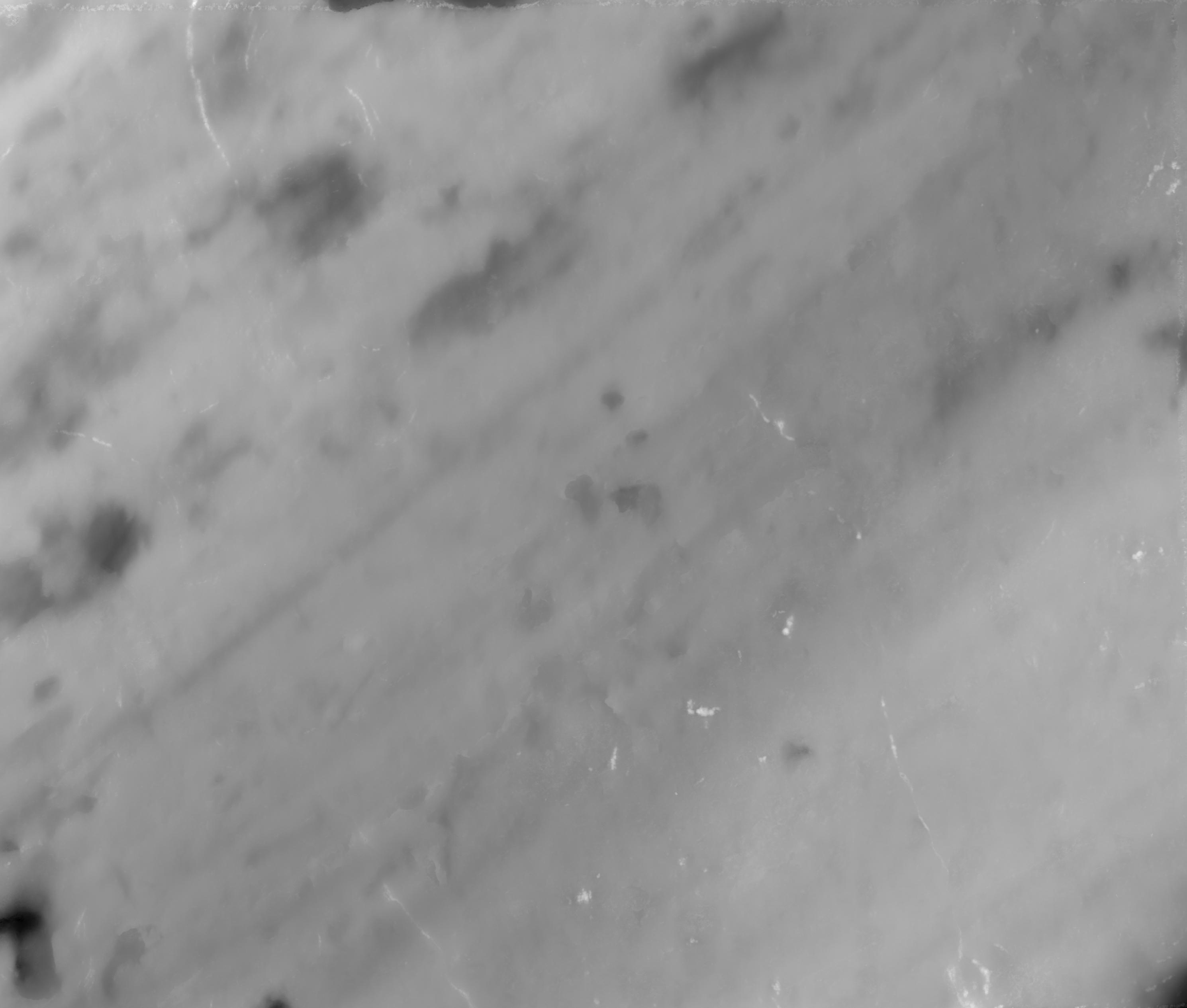}};
\node[yshift=3.8cm] at (Step1) {{\Large \textbf{Input data}}};
\foreach \i in {1,2,3,4,5}{
\node[inner sep = 0pt] (P\i) at (0.7*\i-2.0+10,0.3*\i-1.3) {\includegraphics[width=5cm, height=5cm]{1Dprofile\i}};
}
\node[yshift=4.3cm, xshift=2cm] at (P1) {{\Large \textbf{One-dimensional profiles}}};
\draw[->,line width = 2pt] (Step1) -- node[midway,right,yshift=0.3cm,xshift=-1.5cm] {{\Large Step~\ref{step2} and \ref{step3}}} (6.9,-0.5);
\node[inner sep = 0pt] (Step3) at (10,-8) {\includegraphics[width=8cm, height=6.5cm]{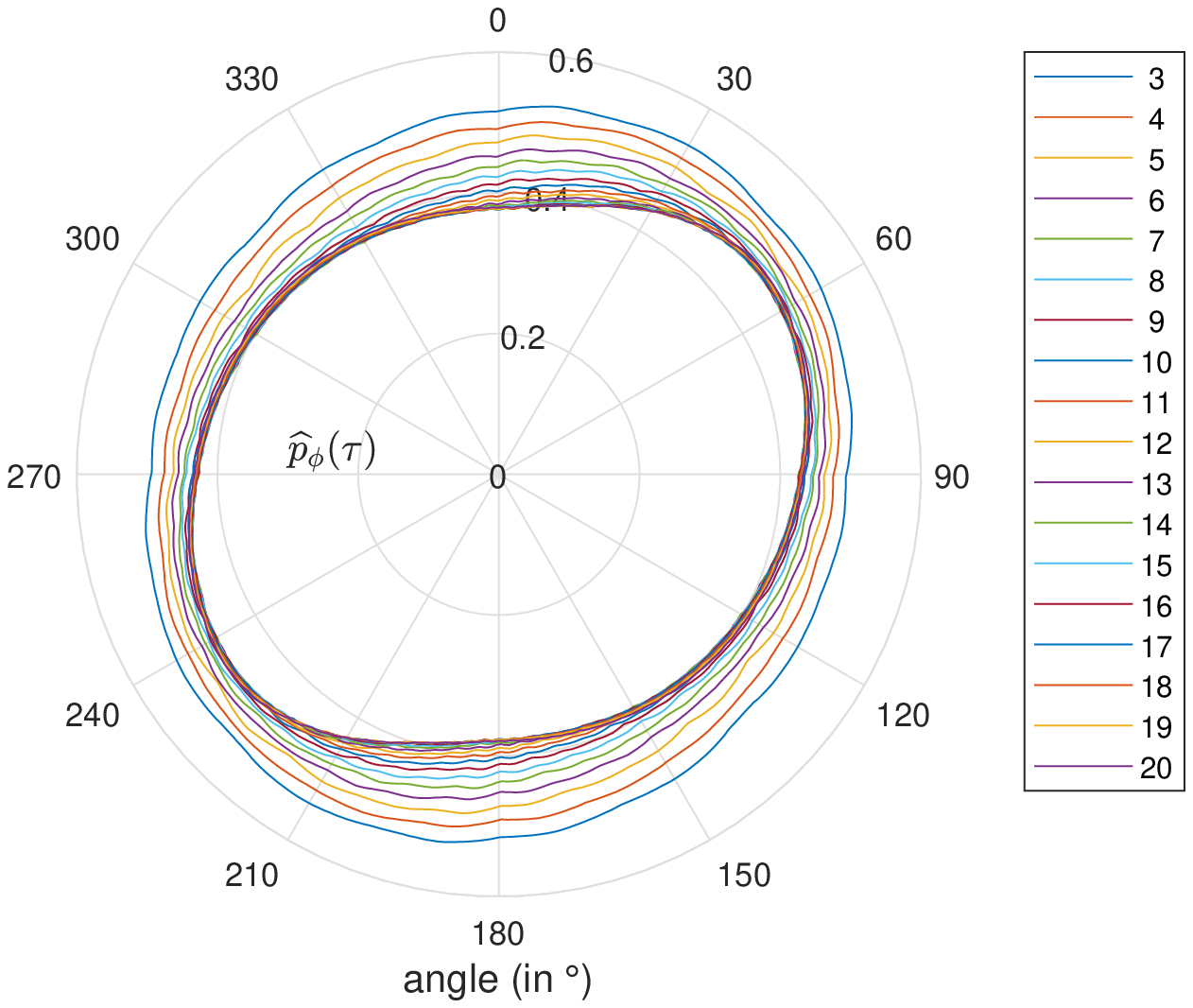}};
\node[yshift=-3.8cm] at (Step3) {{\Large \textbf{Change probabilities}}};
\draw[->,line width = 2pt] (P3)-- node[midway,right] {\Large {Step~\ref{step4}}} (Step3);
\node[inner sep = 0pt] (Step4) at (0,-8) {\includegraphics[width=8cm, height=6.5cm]{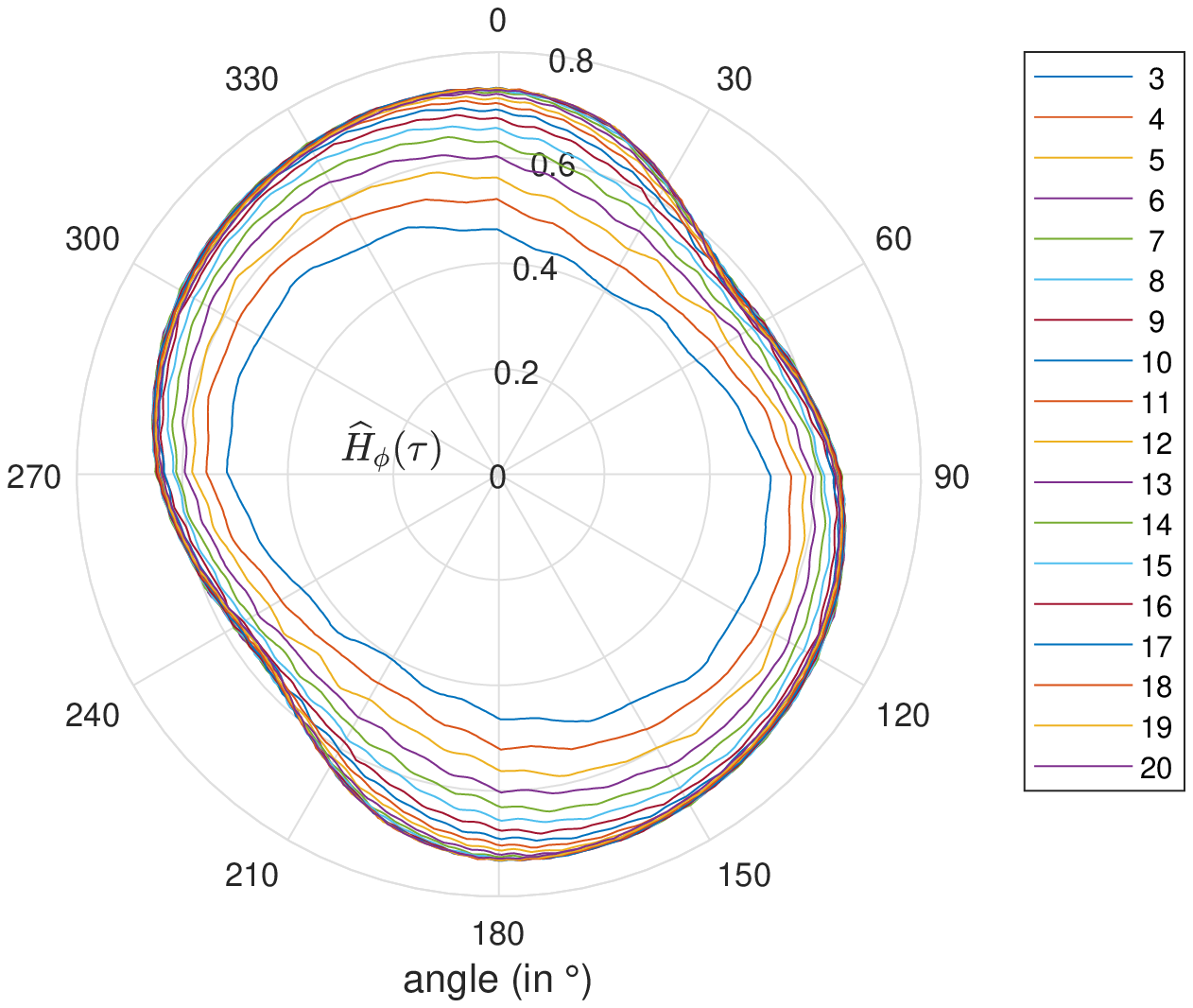}};
\node[yshift=-3.8cm] at (Step4) {{\Large \textbf{Hurst exponent}}};
\draw[->,line width = 2pt] (Step3)-- node[midway,right,yshift=0.3cm,xshift=-0.75cm] {{\Large Step~\ref{step5}}} (Step4);
\end{tikzpicture}
}
\caption{Diagram illustrating the different steps of the algorithm described in Sect.~ \ref{subsec:algorithm}. The polar plots that are created in Step 3 and 4, respectively, show the computed change probability and Hurst exponent as a function of the angle $\phi$ for different delays $\tau$ (represented by different colors)}
\label{fig:steps}
\end{figure}

\begin{enumerate}

\item \label{step2} \textbf{Preparation of the data}: A possible two-dimensional linear trend is removed from the input matrix $A$ by performing a linear regression and subtracting the resulting regression plane from $A$.

\item \label{step3} \textbf{Extraction of one-dimensional profiles}: For each of the $n_\phi$  
    angles $\phi$ create one-dimensional profiles in direction $\phi$ with delay $1$. To achieve this, consider first a line in the image through the origin in direction $\phi$, i.e.\ having angle $\phi$ with the vertical axis (see Fig.~\ref{fig:v(d,phi)}). Take a sequence of equally spaced points on that line which are a pixel width apart from each other.
    For each of these points calculate the height value at the corresponding position using linear interpolation of the given data $A$. Repeat this procedure for all lines intersecting the image that are parallel to the first line and an integer multiple of $\Delta$ apart from it. Each of the one-dimensional profiles extracted in this way will be saved in a vector $\vec{x}^{i}_{\phi} = (x^i_{\phi,t})_{t=0}^{m_{\phi,i}}$ where  $i=1,2,\ldots,k_{\phi,\Delta}$. Here $k_{\phi,\Delta}$ denotes the number of parallel lines intersecting the image for the given $\phi$ and $\Delta$, and $m_{\phi,i}$ +1 is the number of data points in the $i$-th profile.

\item \label{step4} \textbf{Calculation of change probabilities}: For each angle $\phi$ and each delay $\tau$ do the following calculations:
For each profile $i \in \{1,2,\ldots,k_{\phi,\Delta}\}$ with $m_{\phi,i}\geq 2\tau$ compute the number of changes
\begin{align*}
\numChanges_{\phi,\tau,i}:=  \sum_{j=0}^{m_{\phi,i}-2\tau} &\delta(x^i_{\phi,j}<x^i_{\phi,j+\tau} \mbox{ and } x^i_{\phi,j+\tau}\geq x^i_{\phi,j+2\tau})\\
+ &\delta(x^i_{\phi,j}\geq x^i_{\phi,j+\tau} \mbox{ and } x^i_{\phi,j+\tau}< x^i_{\phi,j+2\tau}),
\numberthis\label{eq:numChanges}
\end{align*}
where $\delta(B)=1$ if the assertion $B$ is true and $\delta(B)=0$ otherwise. (For $m_{\phi,i}< 2\tau$, $\numChanges_{\phi,\tau,i}$ is set to $0$.) The estimator $\widehat p_{\phi}(\tau)$ for the change probability is then calculated by
\[ \widehat p_{\phi}(\tau)=\frac{\sum_{i=1}^{k_{\phi,\Delta}} \numChanges_{\phi,\tau,i}}{\sum_{i=1}^{k_{\phi,\Delta}} (m_{\phi,i}-2\tau +1)}. \]


\item \label{step5} \textbf{Estimation of roughness}: For each angle $\phi$ and each delay $\tau$, estimate the roughness $H_{\phi}(\tau)$ of delay $\tau$ in direction $\phi$ by
\begin{equation}
\widehat H_{\phi}(\tau) = h(\widehat p_{\phi}(\tau)),
\label{eq:hatH}
\end{equation}
where $h$ is defined as in \eqref{eq:h}.
\end{enumerate}

\subsection{Output}
The output of the algorithm is a $n_\phi\times n_\tau$ matrix $\widehat{\mathbf{p}}=(\widehat p_{\phi}(\tau))$ containing the estimates for the change probabilities for each angle-delay pair $(\phi,\tau)$. Optional output is a $n_\phi\times n_\tau$ matrix $\widehat{\mathbf{H}}=(\widehat H_{\phi}(\tau))$
containing estimates for the Hurst exponents for each pair $(\phi,\tau)$. Based on this output, the program creates a polar plot of the estimated Hurst exponents $\widehat{H}_{\phi}(\tau)$ as a function of the angle $\phi$ for different delays $\tau$ (see the last step in Fig.~\ref{fig:steps}) and a plot of the estimated Hurst exponents $\widehat{H}_{\phi}(\tau)$ as a function of the delay $\tau$ for different angles $\phi$ (see Fig.~\ref{fig:HTauPlot}).


\subsection{Implementation}\label{sec:implementationDetail}
The delay $\tau \in \N$ should be chosen to be larger than $2$. If the considered angle $\phi$ is not a multiple of $90^\circ$, then the values for the one-dimensional profiles need to be extracted using interpolation. For $\tau\leq 2$, the value of two consecutive points $x^i_{\phi,t}$ and $x^i_{\phi,t+\tau}$ on the profile are so close together that they partially depend on the same values of the original input data. As a result, the interpolated one-dimensional profile will be smoother than the original rock fracture. Hence, the actual roughness will be underestimated. For angles $\phi$ close to a multiple of $90^\circ$ this effect is less prominent.

As an alternative to using linear interpolation for determining intermediate values, one could, for example, use nearest neighbor interpolation. But experiments have shown that the type of interpolation does not have a significant effect on the outcome.

Additionally, notice that the calculations of change probabilities for different angles do not depend on each other and can therefore be processed in parallel.
For the source code of this algorithm see \citeayip{Gutjahr-matlab-code}.

\subsection{Estimating a scale-independent Hurst exponent}
Given the estimators $\widehat H_\phi(\tau)$ for the scale-dependent Hurst exponents, in general, it is not clear what the best way would be to derive from them a single estimator for the (scale-independent) Hurst exponent in a given direction. Considering such a single estimator would, of course, only be justified in the presence of at least some self-affinity. The observation of some range of delays, for which the Hurst exponents do not change significantly, may give rise to the hypothesis that the underlying structure is self-affine for this range of delays. It is then reasonable to expect that the Hurst exponents for such delays best describe the given data. In this situation, the median value of the estimators $\widehat H_\phi(\tau)$ of the scale-dependent Hurst exponents, where $\tau$ varies in the specified range of delays, may be a good estimator for the scale independent Hurst exponent $H_\phi$ (for profiles in direction $\phi$).
The median as an estimator of the Hurst exponent has the additional advantage of being insensitive towards outliers, as can be seen in Fig.~\ref{fig:EffectOfNoise}. The outcome of this method when applied to two different natural fracture surfaces is depicted in Fig.~\ref{fig:medianHurst}.

The median is taken over the values $\widehat H_{\phi}(\tau)$ for $\tau \in \{3,4,\ldots,\tau_{max}\}$, where one needs to specify how to choose $\tau_{max}$. As explained in Sect.~\ref{sec:implementationDetail}, the value of $\widehat H_{\phi}(\tau)$ for $\tau\in\{1,2\}$ is not taken into account.

Clearly, $\tau_{max}$ should depend on the size of the image \revision{and is determined by
\begin{equation}
\tau_{max} = \left\lceil\sqrt[4]{w_X\cdot w_Y}\:\right\rceil,
\label{eq:tauMax}
\end{equation}
where $w_X$ and $w_Y$ are the widths of the image in X- and Y-direction measured in pixels. It is motivated by the need to find a compromise between the pattern length ($2 \tau +1$) and the number of patterns. Increasing the length decreases the number and vice versa.  Choosing for a profile of length $L$ a delay of order $\sqrt{L}$ results roughly in a number of $\sqrt{L}$ non-overlapping patterns. Here, non-overlapping means that the intervals spanned by the pattern locations are disjoint for different patterns. Since in a rectangular image the length $L$ of the profiles changes with the direction, the geometric mean of the profile lengths in X- and Y-direction is chosen as a compromise, resulting in the value $\tau_{max}=\left\lceil\sqrt{\sqrt{w_X\cdot w_Y}}\right\rceil$.}

\section{Application to Natural Fracture Surfaces}\label{sec:experiments}
In this section, the algorithm introduced above is applied to gray level images of two natural fractures, namely a tensile and a shear fracture. The analyzed tensile fracture is a bedding-parallel joint in a carbonate-cemented aeolian sandstone, which opened up within a block specimen sampled from the Schwentesius outcrop in Northern Germany (\cite{fischer_3d_2007}; \cite{heidsiek_small-scale_2020}). Since it can be described as an opening-mode fracture, the surface was not exposed to shear stress. A handheld laser scanner was used to produce a three-dimensional image of the rough fracture surface (Fig.~\ref{fig:originalData:schwentesius}). It is represented by a point cloud, corresponding to a set of data points where the spatial position of each point is uniquely defined by X, Y and Z coordinates. The average point spacing is $0.157$~mm. Further information on the specifications of the laser scanner as well as a photo of the bedding-parallel tensile fracture is provided by \citeayip{hale_method_2020}. The fracture surface does not show a discernible roughness anisotropy except for a barite vein, which cuts through the bedding plane and is noticeable as an elevated ridge-like structure. To make sure that this vein does not skew the result when estimating the roughness of this fracture, the right part of this surface, beginning at $350$~mm, is cut off before applying any methods to it (Fig.~\ref{fig:v(d,phi)}).

Furthermore, a fresh fault slip surface in limestone was used as additional input data for the subsequent roughness analysis (Fig.~\ref{fig:originalData:bolu}). It is part of the Bolu outcrop in Turkey, which is appendant to the North Anatolian fault zone (\cite{candela2012}; \cite{renard_constant_2013}).
The raw data can be found at \citeayip{isterre}. 
In contrast to the tensile fracture (Fig.~\ref{fig:originalData:schwentesius}), the sampled fault surface was exposed to shear stress and has experienced a minimum movement of approximately 20 m (\cite{candela2012}). The fault slip surface shows characteristic features like elongated lenses and linear striae which is also evident from Fig.~3 in \citeayip{candela2012} and Fig.~1 in \citeayip{renard_constant_2013}. For fault slip surfaces, a distinct roughness anisotropy was observed on various length scales, evident from deviating Hurst exponents or fractal dimensions parallel and perpendicular to the slip direction (e.g.,~\cite{candela2012}; \cite{renard_high_2006}; \cite{lee_structural_1996}). The point cloud was produced by a laboratory laser profilometer (Fig.~\ref{fig:originalData:bolu}) and shows a regular point spacing of 20 µm (\cite{candela2012}).

\begin{figure}[ht]
\begin{center}
\begin{subfigure}[c]{0.5\textwidth}
   \includegraphics[width=\textwidth]{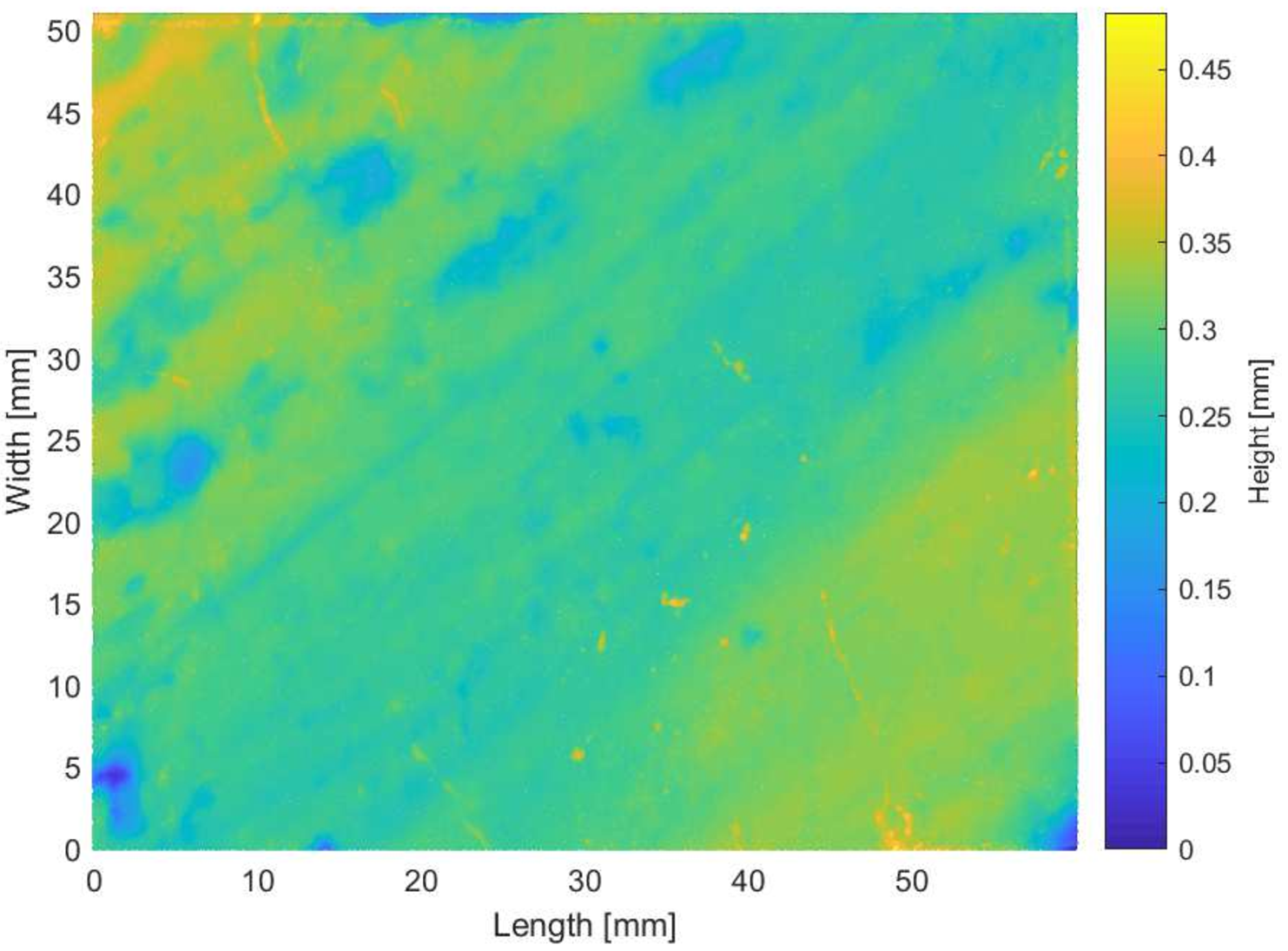}
\caption{}
\label{fig:originalData:bolu}
\end{subfigure}
\end{center}
\begin{subfigure}[c]{\textwidth}
\includegraphics[width=\textwidth]{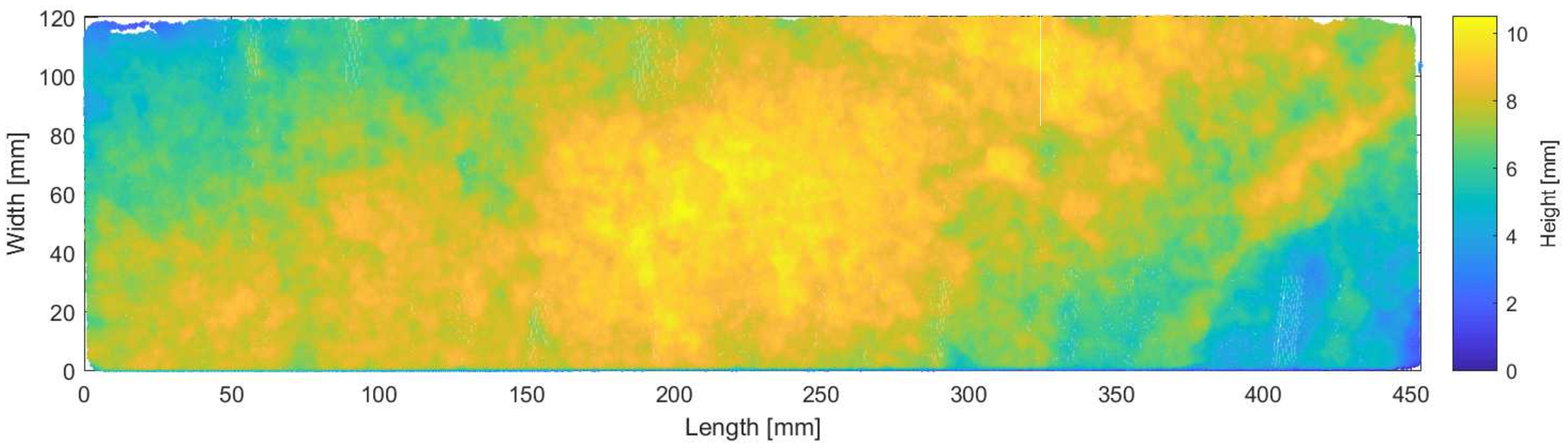}
\caption{}
\label{fig:originalData:schwentesius}
\end{subfigure}

\caption{Point cloud data of (a) a natural shear plane from a limestone outcrop (‘Bolu-1’, \cite{candela2012}) and (b) a natural bedding-parallel tensile fracture in sandstone. The x- and y-axis correspond to the length and width of the captured surface segment, respectively, the z-coordinate, indicated by the color bars, specifies the surface height (all units are in mm).  For the tensile fracture, the z-coordinate (surface height) varies between 0 and 10.53 mm (b), while the surface height variation of the fault plane segment (a) is below $500~\mu$m}
\label{fig:originalData}
\end{figure}

Based on the three-dimensional point cloud data of the fracture surfaces, two-dimensional gray level images are created which are required as input data for the algorithm presented in Sect.~\ref{subsec:algorithm}. First, a surface is fitted to the scattered point data by nearest-neighbor interpolation employing the standard Matlab routine. In this context, a mesh grid has to be specified for evaluating the created interpolant at designated locations. The resolution of the mesh grid should ideally match the resolution or the average point spacing of the point cloud in order to prevent a loss of roughness information. Here, the mesh resolution was set to $0.15$~mm for the tensile fracture and to $0.02$~mm for the fault slip surface. If necessary, the resulting matrix is cropped to remove less dense areas at the edges of the original point cloud. To export the matrix in binary pgm format, the surface height information (i.e. matrix values) is rescaled to fit into the interval $[0,255]$, where $0$ corresponds to the minimum and $255$ to the maximum surface height of the fracture surface.

\subsection{Dependence on the delay}\label{sec:SampleRate}
Figure~\ref{fig:HTauPlot} visualizes how the estimates $\widehat H_{\phi}(\tau)$ of the Hurst exponent change with the delay $\tau$ for fixed angles $\phi$.
Figure~\ref{fig:HTauPlot}a 
indicates that, independent of the angle $\phi$, the values of $\widehat H_{\phi}(\tau)$ increase with $\tau$ until a maximum is reached at around $\tau=19$ and then decrease until around $\tau=80$. In Fig.~\ref{fig:HTauPlot}b,
the values of $\widehat H_{\phi}(\tau)$ decrease with $\tau$ until a minimum is reached near $\tau=40$. The values then slightly increase and the average over all angles stabilizes near $0.4$.

\revision{For comparison, scale-dependent Hurst exponents are also estimated for simulations of fBm, using the same estimator based on change probabilities as for the rock profiles. The results are plotted as a red line in Fig.~\ref{fig:HTauPlot}. The line is derived from $100$ stochastically independent simulations of fBm, whose length  equals the geometric mean of width and height (in pixels) of the corresponding rock surface. The $\tau$-dependent Hurst exponents (see Eq.~\eqref{eq:hatH}) are computed for each simulation and then the average is taken. The Hurst exponent for the simulations was chosen as $0.46$ for the shear plane (left) and $0.51$ for the tensile fracture (right) so that they roughly correspond to the mean of the estimated Hurst exponents for the individual rock surfaces. Since the simulations of fBm are one-dimensional, this estimation does not depend on the angle $\phi$.
}
The observations indicate that for both fracture types the scale-dependent Hurst exponents depend indeed to some extent on the delay.

\begin{figure}[h]
\includegraphics[width=\textwidth,height=8.5cm]{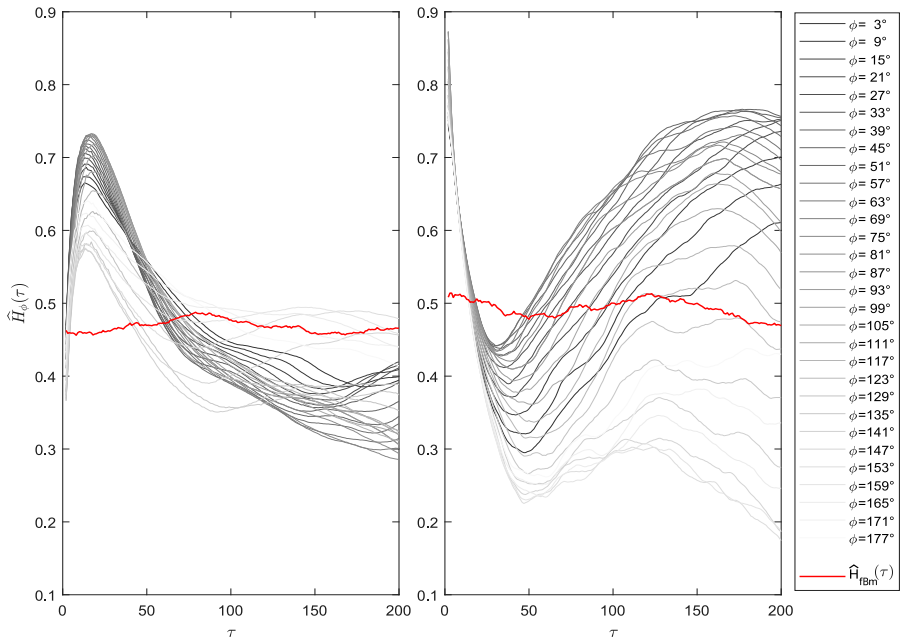}
\caption{Estimates $\widehat H_\phi(\tau)$ of the Hurst exponent (vertical axis) are plotted versus the delay $\tau$ (horizontal axis) for different angles $\phi$ for (left) the shear plane from Bolu outcrop (Fig.~\ref{fig:originalData:bolu}) and (right) the tensile fracture from Schwentesius outcrop (Fig.~\ref{fig:originalData:schwentesius}). Each of the different grey lines corresponds to one out of $30$ angles equally spaced between $0^\circ$ and $180^\circ$.
\revision{The red line $\widehat H_{\mathrm{fBm}}(\tau)$ shows estimates of the $\tau$-dependent Hurst exponents for fBm based on $100$ simulations of fBm}}
\label{fig:HTauPlot}
\end{figure}

\subsection{Comparison to a standard method}\label{subsec:standardMethod}
Figure~\ref{fig:Hcomparison} compares the results of the estimator $\widehat H_{\phi}(\tau)$ for $\tau=3$, $5$ and $25$ with the results provided by a standard method. The standard estimator $\widehat H_{\textrm{wavelet}}$ is based on a method, which is already implemented in the Matlab wavelet-toolbox. This method estimates the variances of wavelets at different detail levels and performs a regression of those variances versus the detail level in a log-log-plot (see e.g.,~\cite{Flandrin1992WaveletAA} and \cite{Abry00self-similarityand} for details). \revision{Technically, the wavelet based estimator only provides an estimate for the Hurst exponent if the underlying process is a fBm. For more general Gaussian processes, one can still calculate the value of $\widehat{H}_{\textrm{wavelet}}$ and it is reasonable to assume that it still represents some measure of roughness, but it is not necessarily equal to the Hurst exponent as defined in \eqref{eq:HurstExp}. It might be possible to interpret $\widehat{H}_{\textrm{wavelet}}$ roughly as some weighted average of the scale dependent Hurst exponents over the available scales, but this is not obvious. Certainly its computation involves information from many delays.}

Figure~\ref{fig:Hcomparison:bolu} shows that the estimates of the Hurst exponent based on $\widehat H_{\phi}(3)$ and $\widehat H_{\textrm{wavelet}}$ are very close to each other for all angles $\phi$, whereas the one based on $\widehat H_\phi(25)$ yields generally larger values than these two. Moreover, the values of $\widehat H_{\phi}(25)$ vary more smoothly with  $\phi$ than the values of $\widehat H_{\phi}(3)$ and $\widehat H_{\textrm{wavelet}}$. All curves have in common that they attain their minimum value for angles around $50^\circ$ and a maximum near $140^\circ$.

In Fig.~\ref{fig:Hcomparison:schwentesius}, the values of $\widehat H_{\textrm{wavelet}}$ are close to those of $\widehat H_\phi(5)$ for most $\phi$. The curves for $\widehat H_\phi(3)$, $\widehat H_\phi(5)$ and $\widehat H_{\textrm{wavelet}}$ are more circular in shape than the ones obtained for the other fracture in Fig.~\ref{fig:Hcomparison:bolu}.

\begin{figure}[ht]
\begin{subfigure}[c]{0.5\textwidth}
\includegraphics[width=\textwidth]{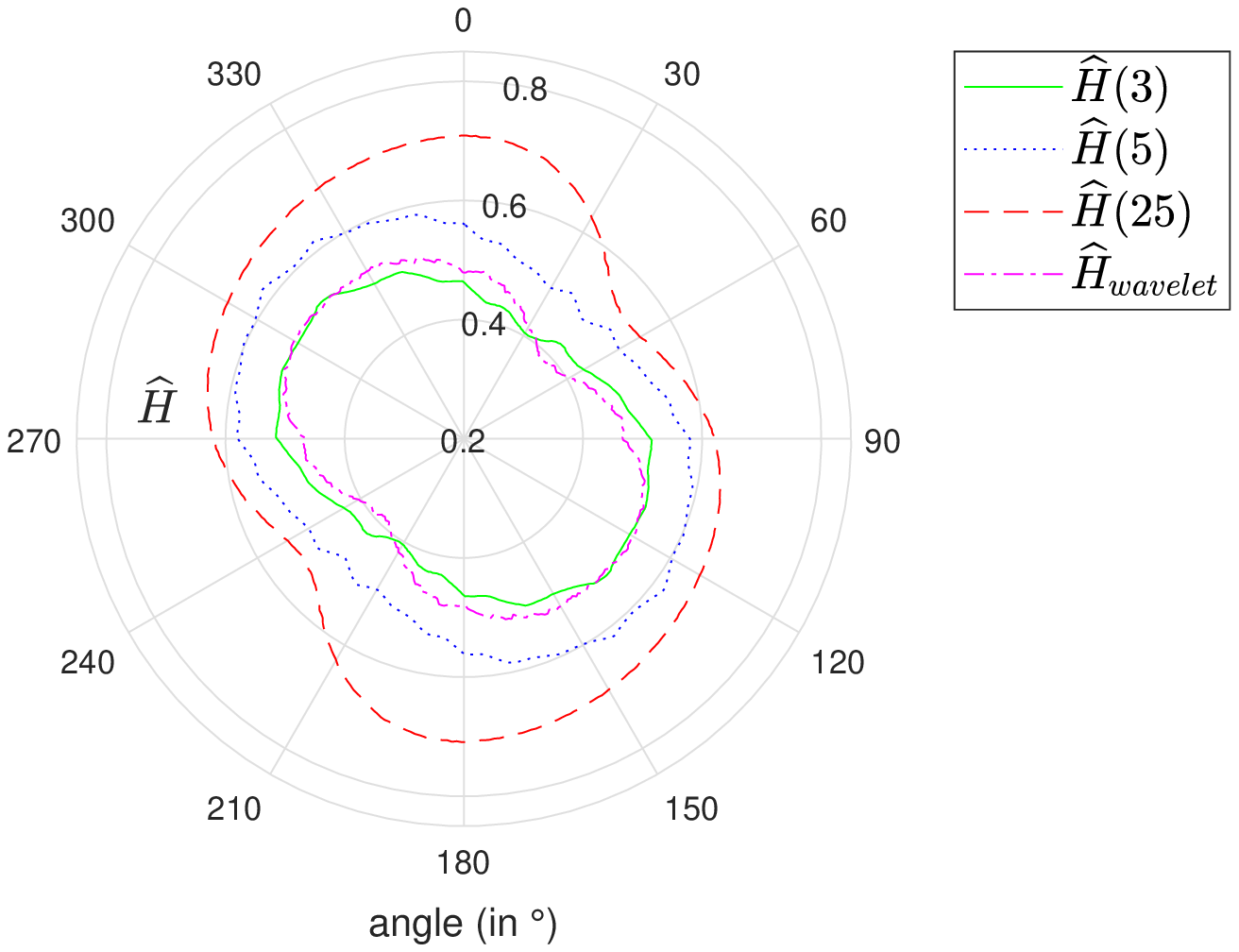}
\caption{}
\label{fig:Hcomparison:bolu}
\end{subfigure}~
\begin{subfigure}[c]{0.5\textwidth}
\includegraphics[width=\textwidth]{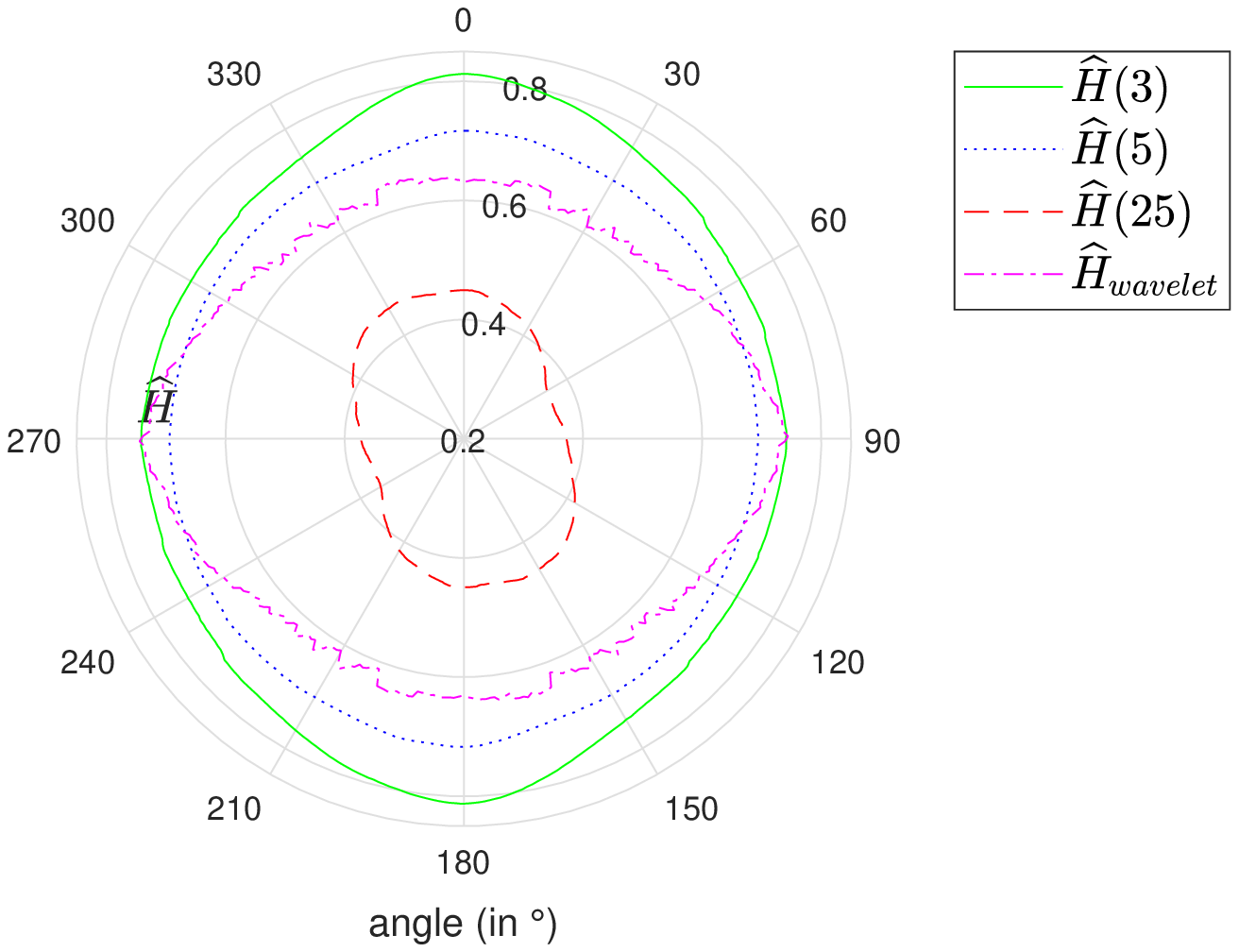}
\caption{}
\label{fig:Hcomparison:schwentesius}
\end{subfigure}
\caption{Radial plots of estimated Hurst exponents as functions of the angle $\phi$ for (a) the shear plane (Fig.~\ref{fig:originalData:bolu}) and (b) the tensile fracture (Fig.~\ref{fig:originalData:schwentesius}). The solid, the dotted and the dashed lines correspond to the estimates based on  $\widehat H_\phi(\tau)$ for the delays $\tau=3$, $5$ and $25$, respectively. The dash-dotted line shows the estimated Hurst exponent $\widehat H_{\textrm{wavelet}}$ using the wavelet based method implemented in Matlab}  
\label{fig:Hcomparison}
\end{figure}

\subsection{Noise effects}

\begin{figure}
\includegraphics[width=\textwidth]{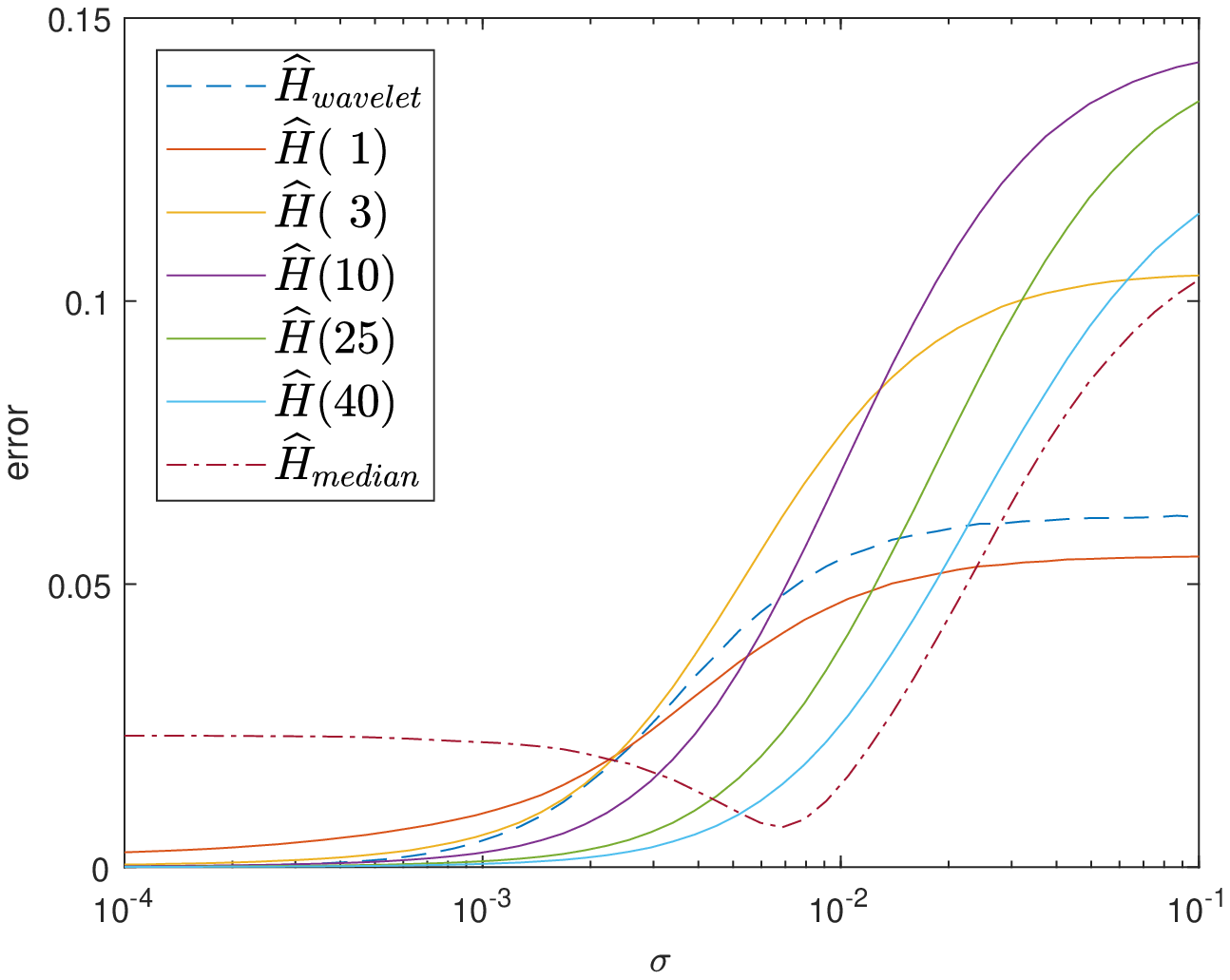}
\caption{
The effect of adding white noise of standard deviation $\sigma$ to the shear plane (Fig.~\ref{fig:originalData:bolu}) is presented for different roughness estimators.
The plots show the error (as defined in 
\eqref{eq:error} with $n_\phi=20$ equally spaced angles) as a function of $\sigma$. 
The dashed line corresponds to the estimator based on wavelets mentioned in Sect.~\ref{subsec:standardMethod}, the dash-dotted line to the median estimator proposed in Sect.~\ref{sec:SampleRate} and the solid lines to the estimators based on change probabilities for different delays}
\label{fig:EffectOfNoise}
\end{figure}

Figure~\ref{fig:EffectOfNoise} shows the effect of noise on the estimators for the Hurst exponent. To simulate this effect, white noise with mean zero and variance $\sigma^2$ was added to each pixel of the image.
The Hurst exponent was then estimated for the noisy image using on the one hand the method implemented in Matlab and on the other hand the new proposed estimators $\widehat H_\phi(\tau)$ for different values of $\tau$. The squares of the differences between the resulting estimates $\widehat H_\phi^\sigma$ and the estimates $\widehat H_\phi$ for the undisturbed image were summed over all considered angles $\phi$. The error is then defined to be the square root of this sum divided by the number $n_\phi$ of different angles
\begin{align}\label{eq:error}
 \textrm{error}=\frac{1}{n_\phi} \sqrt{\sum_{\phi} \left(\widehat H_\phi^\sigma-\widehat H_\phi\right)^2}.
\end{align}
As a reference for the scale of the image values, the square root of the squared average deviation of the pixel values from the mean image value is considered, which is given by
\begin{equation}
\sigma_{img}= \sqrt{\frac{1}{w_x\cdot w_y-1}\sum_{i=1}^{w_x} \sum_{j=1}^{w_y} \left(a_{i,j}-\frac{1}{w_x\cdot w_y}\sum_{i'=1}^{w_x} \sum_{j'=1}^{w_y} a_{i',j'} \right)^2},
\label{eq:stdImage}
\end{equation}
where $(a_{i,j})_{i,j=1}^{w_x,w_y}$ are the pixel values of the image.

In Fig.~\ref{fig:EffectOfNoise}, the effect of noise on different estimators is studied for the shear plane (Fig.~\ref{fig:originalData:bolu}). For this purpose white noise with standard deviation $\sigma$ was added to the surface for different values of $\sigma$ ranging from $10^{-1}$ to $10^{-4}$.
For comparison, the standard deviation over the image as defined in \eqref{eq:stdImage} is approximately equal to $\sigma_{img} \approx 0.064$. For the computation of the error, the profiles in $n_\phi=20$ different directions $\phi$ (chosen equally spaced between $0$ and $\pi$) are taken into account.

\section{Discussion}\label{sec:discussion}
\subsection{Computational Advantages}
From the computational point of view, the estimator $\widehat H_{\phi}(\tau)$ has some advantages compared to alternative methods. Firstly, no linear regression of some logarithmic property is needed to determine $\widehat H_{\phi}(\tau)$. In case the data can be modeled by fBm, the value of $\widehat H_{\phi}(\tau)$ for just one delay $\tau$ is a sufficient estimator for the Hurst exponent. In contrast, for other variance based methods, some property of the data needs to be estimated over many different scales $\tau$ and the Hurst exponent is then obtained as the slope of some linear function of (the logarithm of) $\tau$ determined by this property. Thus, repeated calculations for multiple delays are necessary in this case.

Further, the calculations needed to determine $\widehat p_{\phi}(\tau)$ (and thus $\widehat H_{\phi}(\tau)$) via the number of changes as given by \eqref{eq:numChanges} are mathematically very elementary. One only needs to compare pairs of values, apply some Boolean operations and then count the number of true results. Essentially no costly multiplications are necessary. Hence, the necessary calculations can be done very quickly and, moreover, they can easily be parallelized.

\subsection{Detection of Anisotropy}
Comparing Fig.~\ref{fig:Hcomparison:bolu} and \ref{fig:Hcomparison:schwentesius}, one can observe that the closed lines in the polar plots describing the angle dependent roughness are much more circular in Fig.~\ref{fig:Hcomparison:schwentesius} than in Fig.~\ref{fig:Hcomparison:bolu}. The same can be observed when comparing the median roughness in Fig.~\ref{fig:medianHurst:schwentesius} with the one in Fig.~\ref{fig:medianHurst:bolu}. Therefore, one can conclude that the tensile fracture exhibits less roughness anisotropy than the shear plane. This is in line with the fact that the tensile fracture shows no visible anisotropy, which can be expected due to the formation mechanism.

In contrast, the existing roughness anisotropy of the shear plane was detected with the methods introduced here.

\subsection{Dependence on the Delay}
Figure~\ref{fig:HTauPlot} shows that the value of the $\widehat H_{\phi}(\tau)$ is not constant over all delays $\tau$. \revision{One could ask whether this is merely a statistical phenomenon caused by fluctuating estimates. However, the clear monotonic behavior of the estimates as a function of $\tau$ contradicts such a hypothesis. Comparing the fluctuations of $\widehat H_{\phi}(\tau)$ to the fluctuations of corresponding estimates $\widehat H_{\mathrm{fBm}}(\tau)$ of fBm (see  the red line in Fig.~\ref{fig:HTauPlot}) further contradicts this hypothesis. 
Already a rough visual comparison shows that $\widehat H_{\mathrm{fBm}}(\tau)$ and $\widehat H_{\phi}(\tau)$ behave significantly different. Additionally, the relatively small fluctuations of the red line confirm that for fBm the $\tau$-dependent Hurst exponents do not depend on the delay $\tau$ as explained in Sect.~\ref{sec:fbm}. The dependence of the estimated Hurst exponents $\widehat H_{\phi}(\tau)$ on the delay indicates that modeling of the considered data by fBm only makes sense to a limited extent.}

The delays $\tau=19$ in Fig.~\ref{fig:HTauPlot}a 
and $\tau=40$ in Fig.~\ref{fig:HTauPlot}b, 
for which $\widehat H_{\phi}(\tau)$ is maximal or minimal, respectively, might correspond to the length of some characteristic structure on the fracture surface. The clear maximum visible in Fig.~\ref{fig:HTauPlot} could, potentially, be used to identify such characteristic lengths.

For larger delays $\tau$, the value of $\widehat H_{\phi}(\tau)$ seems to fluctuate less and stabilize around some constant value depending on the angle $\phi$. This observation is consistent with the theoretical results in Sect.~\ref{sec:asymptoticProperties}.

Additionally, notice that $\widehat H_{\phi}(\tau)$ is, on average, tending towards values around $0.5$ for large $\tau$ in Fig.s~\ref{fig:HTauPlot}a 
and \ref{fig:HTauPlot}b. 
Recall that data with independent increments can be modeled by a fBm with Hurst exponent $H=0.5$. The above observation could therefore be explained by the fact that the stochastic dependence between points on the fracture surface, which are separated by a large distance $\tau$, decreases faster than predicted by the mathematical model for the fracture surface.

\begin{figure}[ht]
\begin{subfigure}[c]{0.48\textwidth}
\centering
\includegraphics[width=\textwidth]{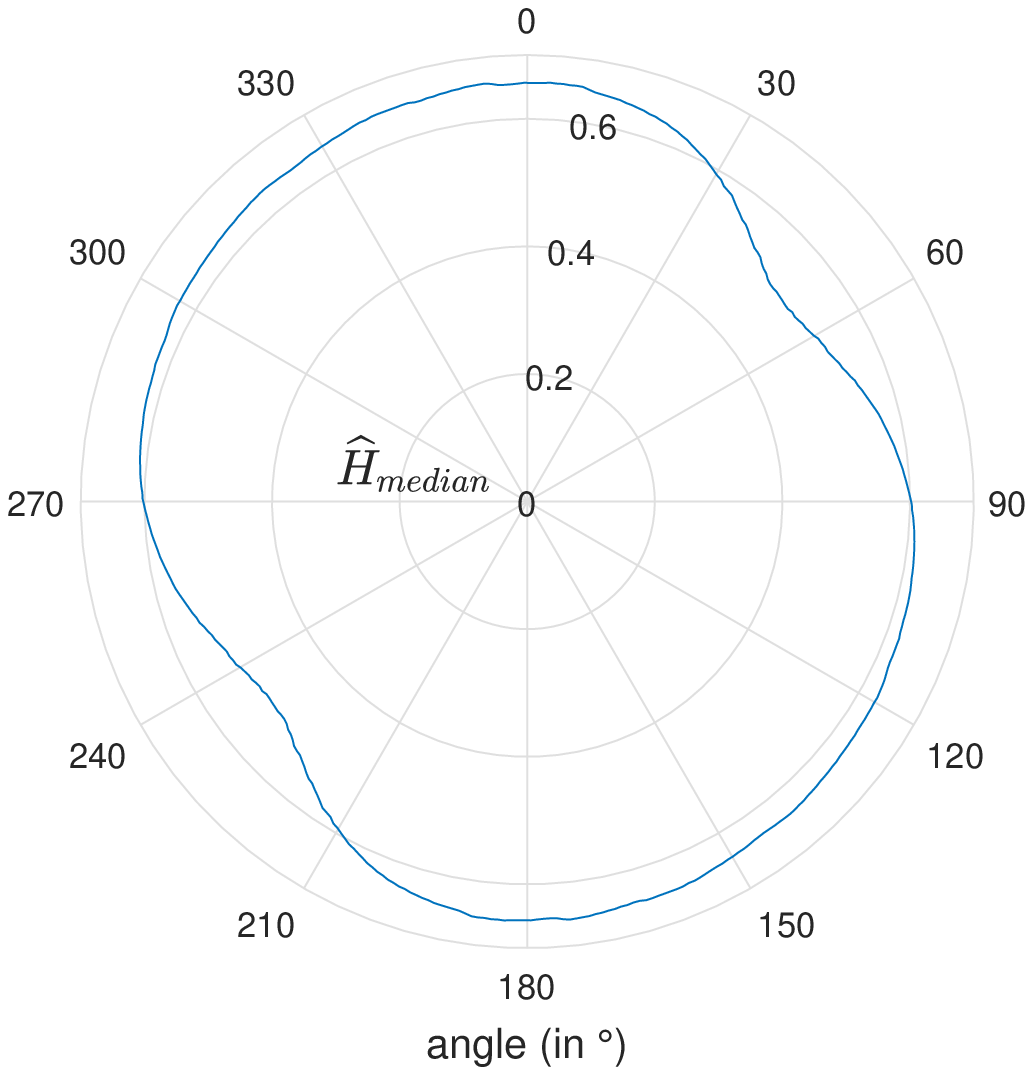}
\caption{Shear fracture ($\tau_{max}=53$)}
\label{fig:medianHurst:bolu}
\end{subfigure}~
\begin{subfigure}[c]{0.48\textwidth}
\centering
\includegraphics[width=\textwidth]{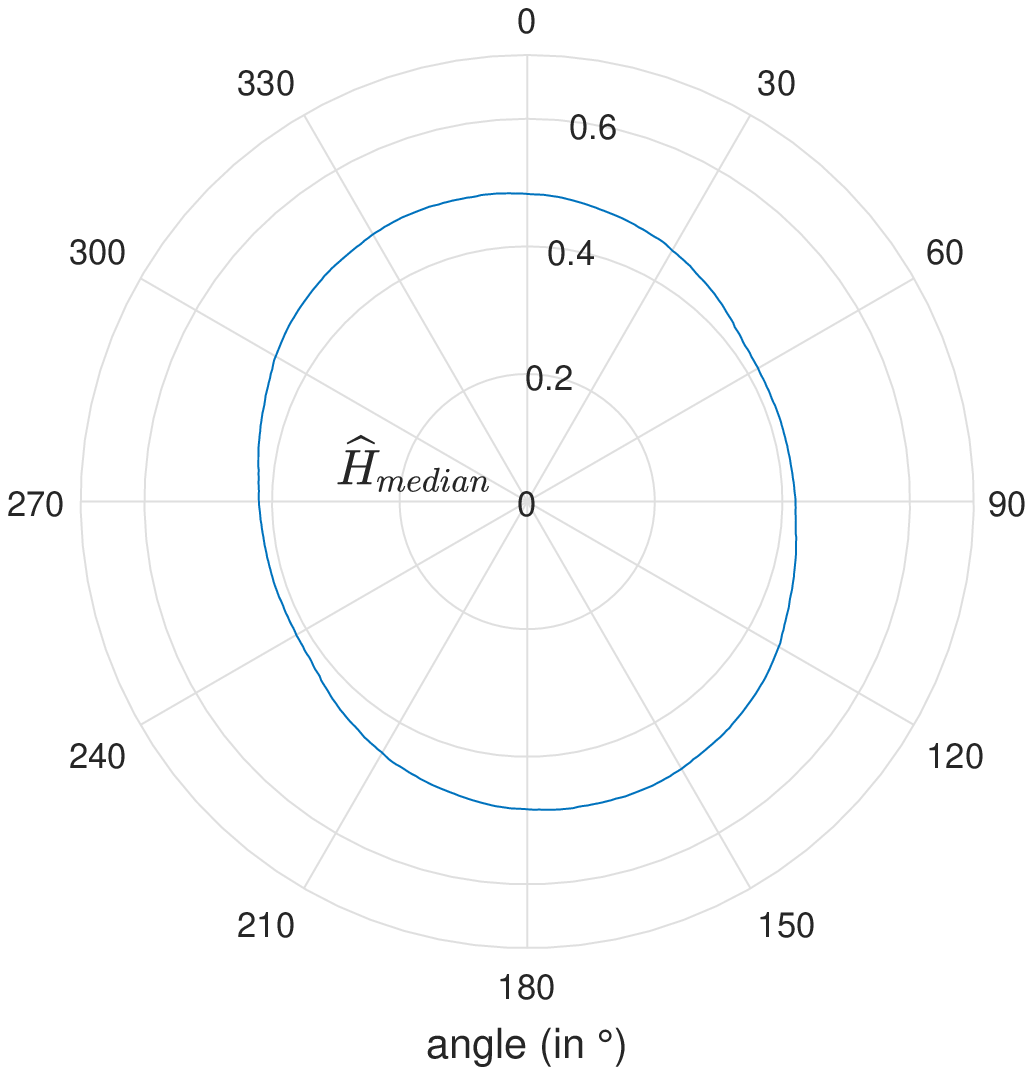}
\caption{Tensile fracture ($\tau_{max}=37$)}
\label{fig:medianHurst:schwentesius}
\end{subfigure}
\caption{Polar plots showing for the two natural fracture surfaces (Fig.~\ref{fig:originalData}) the \emph{median Hurst exponent} on the radial axis in dependence of the angle $\phi$. It was calculated by taking for each angle $\phi$ the median value of the change probabilities for integer delays $\tau$ between $3$ and $\tau_{max}$ and then applying formula~\eqref{eq:h} to this median. $\tau_{max}$ was determined according to Eq.~\eqref{eq:tauMax}}
\label{fig:medianHurst}
\end{figure}

\subsection{Noise effects}\label{sec:noise}
Figure~\ref{fig:EffectOfNoise} shows that the larger the sampling rate $\tau$, the lesser the estimator $\widehat H_{\phi}(\tau)$ is affected by noise. 
This is not surprising. Since the difference between values $x_{t}$ and $x_{t+\tau}$ is typically the smaller on average the smaller the $\tau$, small random perturbations are more likely to cause a change of the pattern of some $(x_t,x_{t+\tau},x_{t+2\tau})$ for smaller $\tau$.
Furthermore, one can see that the estimator $ \widehat H_{\phi}(\tau)$ is less affected by noise than the wavelet based  estimator $\widehat H_{\textrm{wavelet}}$, provided that the strength of the noise is not too large. The reason for this is that small random perturbations of given data $(x_t)_{t=1}^n$ will only significantly effect the estimator $\widehat H_{\phi}(\tau)$ if the randomness changes the ordering of $x_{t},x_{t+\tau},x_{t+2\tau}$ for sufficiently many $t$. 

\subsection{Comparison to Standard Method}
In Fig.~\ref{fig:Hcomparison}, one can see that for specific sampling rates $\tau$ the estimator $\widehat H_\phi(\tau)$ yields results that are comparable to standard methods. Figure~\ref{fig:Hcomparison:bolu} shows that, even though the value of $\widehat H_\phi(\tau)$ is different for larger delays $\tau$, the angles that yield the minimal and maximal Hurst exponent remain unchanged. This indicates that the structure of the rock surface that is causing the anisotropy is invariant over a specific range of scales.

One can additionally observe the greater smoothness of $\widehat H_\phi(25)$ compared to $\widehat H_\phi(3)$, which is in line with the effects described in Sect.~\ref{sec:noise}. For the same reason, the angle at which the minimum of $\widehat H_\phi(25)$ over all angles is found, is more pronounced. Hence, the delay $\tau$ should not be chosen too small in general.

For the exemplary fractures considered here, $\tau=5$ is considered reasonable. Since it is difficult to know in advance which value for $\tau$ is a good choice, calculating the median fracture roughness over many delays $\tau$ as shown in Fig.~\ref{fig:medianHurst} is appropriate in these cases.

\section{Conclusion} \label{sec:conclusion}
The proposed concept of scale-dependent Hurst exponents based on change probabilities  provides scale and direction dependent roughness information for rough surfaces beyond the summarizing information provided by classical (single valued) roughness measures.
Beside some stationarity, the new method does not require any additional assumption on the studied surface or its directional profiles such as self-affinity or scale-invariance and may therefore be applied in situations when, strictly speaking, existing methods are not applicable because they are based on such assumptions.

In the presence of self-affinity, there is a direct relation to the classical Hurst exponent. As is discussed in Theorem 1, an asymptotic relation between scale-dependent and classical Hurst exponents can even be established for a much wider
class of processes than just fBm. This suggests to use
change probabilities (if not as a new independent roughness measure) as a
means of estimating the classical Hurst exponent in such situations.

It was demonstrated that estimators for roughness based on change probabilities yield results similar to alternative methods,  while having computational advantages and being less sensitive to noise.

On the practical side, by exploring two natural fracture surfaces, it was illustrated how the methods proposed here can be used to quantify fracture roughness and to identify roughness anisotropy.  In particular,  these methods allow to investigate how fracture roughness varies across different scales. This is useful, for example, for checking the hypothesis of self-affinity.

In the self-affine case, all estimated scale-dependent Hurst exponents defined on the base of change probabilities are expected to be near to the true Hurst exponent of the structure. Otherwise, in the absence of self-affinity or milder adequate premises, estimators of a possibly undefined Hurst exponent can yield diverse values that are difficult to interpret. This may even result in conflicting statements about the direction of maximum roughness for anisotropic rock surfaces (i.e. shear or fault surfaces). In this case, the use of the scale-dependent Hurst exponents could possibly help to find anisotropies, while their detection by only one (averaging) quantifier is impossible.

The proposed method may also be useful for extracting additional scale-dependent information about the structure of a fracture surface such as the characteristic size of building blocks of a material.




\bibliographystyle{MG}       
{\footnotesize
\bibliography{bibliography}}   

\begin{thebibliography}{67}
\providecommand{\natexlab}[1]{#1}

\bibitem[Abry et~al.(2000)]{Abry00self-similarityand}
Abry P, Flandrin P, Taqqu M~S, Veitch D (2000) Self-similarity and long-range
  dependence through the wavelet lens. In Theory and Applications of Long-Range
  Dependence, Birkh{\"a}user Boston, 527--556

\bibitem[{Amig\'o}(2010)]{zbMATH05702310}
{Amig\'o} J~M (2010) {Permutation complexity in dynamical systems. Ordinal
  patterns, permutation entropy and all that.} Dordrecht: Springer, ISBN
  978-3-642-04083-2/hbk; 978-3-642-04084-9/ebook

\bibitem[{Amig\'o} et~al.(2014)]{AmigoKellerUnakafova2014}
{Amig\'o} J~M, {Keller} K, {Unakafova} V~A (2014) {Ordinal symbolic analysis
  and its application to biomedical recordings}. Philosophical Transactions of
  the Royal Society of London Series A 373(2034):20140091--20140091

\bibitem[{Arcones}(1994)]{zbMATH00828385}
{Arcones} M~A (1994) {Limit theorems for nonlinear functionals of a stationary
  Gaussian sequence of vectors.} {Ann Probab} 22(4):2242--2274, ISSN 0091-1798;
  2168-894X/e

\bibitem[Bandis et~al.(1983)]{bandis_fundamentals_1983}
Bandis S, Lumsden A, Barton N (1983) Fundamentals of rock joint deformation.
  International Journal of Rock Mechanics and Mining Sciences \& Geomechanics
  Abstracts 20(6):249--268, ISSN 01489062

\bibitem[{Bandt} and {Shiha}(2007)]{zbMATH05290232}
{Bandt} C, {Shiha} F (2007) {Order patterns in time series.} {J Time Ser Anal}
  28(5):646--665, ISSN 0143-9782; 1467-9892/e

\bibitem[Barton(1973)]{barton_review_1973}
Barton N (1973) Review of a new shear-strength criterion for rock joints.
  Engineering Geology 7(4):287--332

\bibitem[Barton and Choubey(1977)]{barton_shear_1977}
Barton N, Choubey V (1977) The shear strength of rock joints in theory and
  practice. Rock Mechanics 10(1):1--54, ISSN 0035-7448, 1434-453X

\bibitem[Bistacci et~al.(2011)]{bistacci_fault_2011}
Bistacci A, Griffith W~A, Smith S~A~F, Di~Toro G, Jones R, Nielsen S (2011)
  Fault roughness at seismogenic depths from {LIDAR} and photogrammetric
  analysis. Pure and Applied Geophysics 168:2345--2363

\bibitem[Boutt et~al.(2006)]{boutt_trapping_2006}
Boutt D~F, Grasselli G, Fredrich J~T, Cook B~K, Williams J~R (2006) Trapping
  zones: The effect of fracture roughness on the directional anisotropy of
  fluid flow and colloid transport in a single fracture. Geophysical Research
  Letters 33(21):L21402, ISSN 0094-8276

\bibitem[Brodsky et~al.(2016)]{brodsky_constraints_2016}
Brodsky E~E, Kirkpatrick J~D, Candela T (2016) Constraints from fault roughness
  on the scale-dependent strength of rocks. Geology 44(1):19--22, ISSN
  0091-7613, 1943-2682

\bibitem[Brown(1987)]{brown_fluid_1987}
Brown S~R (1987) Fluid flow through rock joints: the effect of surface
  roughness. Journal of Geophysical Research: Solid Earth 92:1337--1347,
  {ISBN}: 0148-0227 Publisher: Wiley Online Library

\bibitem[Candela et~al.(2012)]{candela2012}
Candela T, Renard F, Klinger Y, Mair K, Schmittbuhl J, Brodsky E~E (2012)
  Roughness of fault surfaces over nine decades of length scales. Journal of
  Geophysical Research: Solid Earth 117(B8)

\bibitem[Chang et~al.(1951)]{ChangPielEssigmann51}
Chang S, Pihl G, Essigmann M (1951) Representations of speech sounds and some
  of their statistical properties. Proc IRE 39:147--153

\bibitem[Coeurjolly(2000)]{coeurjolly:hal-00383120}
Coeurjolly J~F (2000) {Simulation and identification of the fractional Brownian
  motion: a bibliographical and comparative study}. {J Stat Softw} 5(7):1--53

\bibitem[Corradetti et~al.(2017)]{corradetti_evaluating_2017}
Corradetti A, {McCaffrey} K, De~Paola N, Tavani S (2017) Evaluating roughness
  scaling properties of natural active fault surfaces by means of multi-view
  photogrammetry. Tectonophysics 717:599--606, ISSN 00401951

\bibitem[Corradetti et~al.(2020)]{corradetti_impact_2020}
Corradetti A, Zambrano M, Tavani S, Tondi E, Seers T~D (2020) The impact of
  weathering upon the roughness characteristics of a splay of the active fault
  system responsible for the massive 2016 seismic sequence of the central
  apennines, italy. Geological Society of America Bulletin 130

\bibitem[Dong and Ju(2020)]{dong_quantitative_2020}
Dong J, Ju Y (2020) Quantitative characterization of single-phase flow through
  rough-walled fractures with variable apertures. Geomechanics and Geophysics
  for Geo-Energy and Geo-Resources 6:42, ISSN 2363-8419, 2363-8427

\bibitem[Doob(1935)]{deltaMethod}
Doob J (1935) {The Limiting Distributions of Certain Statistics}. The Annals of
  Mathematical Statistics 6(3):160 -- 169

\bibitem[Dou et~al.(2019)]{dou_multiscale_2019}
Dou Z, Sleep B, Zhan H, Zhou Z, Wang J (2019) Multiscale roughness influence on
  conservative solute transport in self-affine fractures. International Journal
  of Heat and Mass Transfer 133:606--618

\bibitem[El-Soudani(1978)]{el-soudani_profilometric_1978}
El-Soudani S~M (1978) Profilometric analysis of fractures. Metallography
  11(3):247--336

\bibitem[Ewing and Taylor(1969)]{EwingTaylor69}
Ewing G, Taylor J (1969) Computer recognition of speech using zero-crossing
  information. IEEE Trans Audio Electroacoust 17:37--40

\bibitem[Fischer et~al.(2007)]{fischer_3d_2007}
Fischer C, Gaupp R, Dimke M, Sill O (2007) A 3d high resolution model of
  bounding surfaces in aeolian-fluvial deposits: an outcrop analogue study from
  the permian rotliegend, northern germany. Journal of Petroleum Geology
  30(3):257--273, ISSN 0141-6421, 1747-5457

\bibitem[Flandrin(1992)]{Flandrin1992WaveletAA}
Flandrin P (1992) Wavelet analysis and synthesis of fractional brownian motion.
  IEEE Trans Inf Theory 38:910--917

\bibitem[Gneiting and Schlather(2004)]{Gneiting2004}
Gneiting T, Schlather M (2004) Stochastic models that separate fractal
  dimension and the hurst effect. SIAM Review 46(2):269--282, ISSN 00361445

\bibitem[Gutjahr(2021)]{Gutjahr-matlab-code}
Gutjahr T (2021) Hurst exponent based on change probabilities. {MATLAB} central
  file exchange.
  \url{https://www.mathworks.com/matlabcentral/fileexchange/80008-hurst-exponent-based-on-change-probabilities},
  accessed: 2021-02-25

\bibitem[Hale et~al.(2020)]{hale_method_2020}
Hale S, Naab C, Butscher C, Blum P (2020) Method comparison to determine
  hydraulic apertures of natural fractures. Rock Mechanics and Rock Engineering
  53(3):1467--1476, ISSN 0723-2632, 1434-453X

\bibitem[Heidsiek et~al.(2020)]{heidsiek_small-scale_2020}
Heidsiek M, Butscher C, Blum P, Fischer C (2020) Small-scale diagenetic facies
  heterogeneity controls porosity and permeability pattern in reservoir
  sandstones. Environmental Earth Sciences 79(18):425, ISSN 1866-6280,
  1866-6299

\bibitem[Ho and Sun(1987)]{HO1987144}
Ho H~C, Sun T~C (1987) A central limit theorem for non-instantaneous filters of
  a stationary gaussian process. Journal of Multivariate Analysis 22(1):144 --
  155, ISSN 0047-259X

\bibitem[Hsiung et~al.(1995)]{rock}
Hsiung S~M, Ghosh A, Chowdhury A~H (1995) On natural rock joint profile
  characterization using self-affine fractal approach. In Daemen J, Schultz R,
  editors, Proceedings of the 35th U.S. Symposium on Rock Mechanics, Balkema,
  Rotterdam, 681--687

\bibitem[Huang et~al.(1992)]{huang_applicability_1992}
Huang S~L, Oelfke S~M, Speck R~C (1992) Applicability of fractal
  characterization and modelling to rock joint profiles. International Journal
  of Rock Mechanics and Mining Sciences \& Geomechanics Abstracts 29(2):89--98

\bibitem[Issa et~al.(2003)]{issa_fractal_2003}
Issa M~A, Issa M~A, Islam M~S, Chudnovsky A (2003) Fractal dimension––a
  measure of fracture roughness and toughness of concrete. Engineering Fracture
  Mechanics 70(1):125--137, ISSN 0013-7944

\bibitem[ISTerre(2021)]{isterre}
ISTerre (2021) Institut des sciences de la terre. {F}ault slip surface data.
  \url{https://www.isterre.fr/french/recherche-observation/equipes/mecanique-des-failles/moyens-et-outils/article/donnees.html},
  accessed: 2021-02-25

\bibitem[Jiang et~al.(2006)]{jiang_estimating_2006}
Jiang Y, Li B, Tanabashi Y (2006) Estimating the relation between surface
  roughness and mechanical properties of rock joints. International Journal of
  Rock Mechanics and Mining Sciences 43(6):837--846, ISSN 13651609

\bibitem[Kottwitz et~al.(2020)]{kottwitz_hydraulic_2020}
Kottwitz M~O, Popov A~A, Baumann T~S, Kaus B~J~P (2020) The hydraulic
  efficiency of single fractures: correcting the cubic law parameterization for
  self-affine surface roughness and fracture closure. Solid Earth
  11(3):947--957, ISSN 1869-9529

\bibitem[Lee and Bruhn(1996)]{lee_structural_1996}
Lee J~J, Bruhn R~L (1996) Structural anisotropy of normal fault surfaces.
  Journal of Structural Geology 18(8):1043--1059, ISSN 01918141

\bibitem[Li et~al.(2019)]{li_stochastic_2019}
Li X, Jiang Z, Couples G~G (2019) A stochastic method for modelling the
  geometry of a single fracture: spatially controlled distributions of
  aperture, roughness and anisotropy. Transport in Porous Media
  128(2):797--819, ISSN 0169-3913, 1573-1634

\bibitem[Liu et~al.(2020)]{liu_three-dimensional_2020}
Liu R, He M, Huang N, Jiang Y, Yu L (2020) Three-dimensional
  double-rough-walled modeling of fluid flow through self-affine shear
  fractures. Journal of Rock Mechanics and Geotechnical Engineering
  12(1):41--49, ISSN 16747755

\bibitem[Magsipoc et~al.(2020)]{magsipoc_2d_2020}
Magsipoc E, Zhao Q, Grasselli G (2020) 2d and 3d roughness characterization.
  Rock Mechanics and Rock Engineering 53(3):1495--1519, ISSN 0723-2632,
  1434-453X

\bibitem[Malinverno(1990)]{malinverno_simple_1990}
Malinverno A (1990) A simple method to estimate the fractal dimension of
  self-affine series. Geophysical Research Letters 17(11):1953--1956

\bibitem[Mandelbrot(1985)]{mandelbrot_self-affine_1985}
Mandelbrot B~B (1985) Self-affine fractals and fractal dimension. Physica
  Scripta 32(4):257--260

\bibitem[{McClean} and Evans(2002)]{mcclean_apparent_2002}
{McClean} C~J, Evans I~S (2002) Apparent fractal dimensions from continental
  scale digital elevation models using variogram methods. Transactions in {GIS}
  4(4):361--378

\bibitem[Morgan et~al.(2013)]{morgan_cracking_2013}
Morgan S~P, Johnson C~A, Einstein H~H (2013) Cracking processes in barre
  granite: fracture process zones and crack coalescence. International Journal
  of Fracture 180(2):177--204, ISSN 0376-9429, 1573-2673

\bibitem[Myers(1962)]{myers_characterization_1962}
Myers N~O (1962) Characterization of surface roughness. Wear 5(3):182--189

\bibitem[Odling(1994)]{odling_natural_1994}
Odling N~E (1994) Natural fracture profiles, fractal dimension and joint
  roughness coefficients. Rock Mechanics and Rock Engineering 27(3):135--153,
  ISSN 0723-2632, 1434-453X

\bibitem[Poon et~al.(1992)]{poon_surface_1992}
Poon C~Y, Sayles R~S, Jones T~A (1992) Surface measurement and fractal
  characterization of naturally fractured rocks. Journal of Physics D: Applied
  Physics 25(8):1269--1275, ISSN 0022-3727, 1361-6463

\bibitem[Power and Durham(1997)]{power_topography_1997}
Power W~L, Durham W~B (1997) Topography of natural and artificial fractures in
  granitic rocks: Implications for studies of rock friction and fluid
  migration. International Journal of Rock Mechanics and Mining Sciences
  34(6):979--989

\bibitem[Power and Tullis(1991)]{power_euclidean_1991}
Power W~L, Tullis T~E (1991) Euclidean and fractal models for the description
  of rock surface roughness. Journal of Geophysical Research: Solid Earth
  96:415--424

\bibitem[Renard et~al.(2013)]{renard_constant_2013}
Renard F, Candela T, Bouchaud E (2013) Constant dimensionality of fault
  roughness from the scale of micro-fractures to the scale of continents.
  Geophysical Research Letters 40(1):83--87, ISSN 00948276

\bibitem[Renard et~al.(2006)]{renard_high_2006}
Renard F, Voisin C, Marsan D, Schmittbuhl J (2006) High resolution 3d laser
  scanner measurements of a strike-slip fault quantify its morphological
  anisotropy at all scales. Geophysical Research Letters 33(4):L04305, ISSN
  0094-8276

\bibitem[Roux et~al.(1993)]{roux_physical_1993}
Roux S, Schmittbuhl J, Vilotte J~P, Hansen A (1993) Some physical properties of
  self-affine rough surfaces. Europhysics Letters 23(4):277--282

\bibitem[Sagy et~al.(2007)]{sagy_evolution_2007}
Sagy A, Brodsky E~E, Axen G~J (2007) Evolution of fault-surface roughness with
  slip. Geology 35(3):283--286

\bibitem[Schmittbuhl et~al.(1993)]{schmittbuhl_field_1993}
Schmittbuhl J, Gentier S, Roux S (1993) Field measurements of the roughness of
  fault surfaces. Geophysical Research Letters 20(8):639--641

\bibitem[Seybold et~al.(2020)]{seybold_flow_2020}
Seybold H~J, Carmona H~A, Filho F~A~L, Araújo A~D, Filho F~N, Andrade J~S
  (2020) Flow through three-dimensional self-affine fractures. Physical Review
  Fluids 5(10):104101, ISSN 2469-990X

\bibitem[Siman-Tov et~al.(2013)]{siman-tov_nanograins_2013}
Siman-Tov S, Aharonov E, Sagy A, Emmanuel S (2013) Nanograins form carbonate
  fault mirrors. Geology 41(6):703--706

\bibitem[Sinn and Keller(2011{\natexlab{a}})]{doi:10.1137/S0040585X97984991}
Sinn M, Keller K (2011{\natexlab{a}}) Covariances of zero crossings in gaussian
  processes. Theory of Probability \& Its Applications 55(3):485--504

\bibitem[Sinn and Keller(2011{\natexlab{b}})]{Sinn2011}
Sinn M, Keller K (2011{\natexlab{b}}) {Estimation of ordinal pattern
  probabilities in Gaussian processes with stationary increments}.
  Computational Statistics \& Data Analysis 55(4):1781--1790

\bibitem[Stigsson and Mas~Ivars(2019)]{stigsson_novel_2019}
Stigsson M, Mas~Ivars D (2019) A novel conceptual approach to objectively
  determine {JRC} using fractal dimension and asperity distribution of mapped
  fracture traces. Rock Mechanics and Rock Engineering 52(4):1041--1054, ISSN
  0723-2632, 1434-453X

\bibitem[Thompson and Brown(1991)]{thompson_effect_1991}
Thompson M~E, Brown S~R (1991) The effect of anisotropic surface roughness on
  flow and transport in fractures. Journal of Geophysical Research: Solid Earth
  96:21923--21932

\bibitem[Tsang and Witherspoon(1983)]{tsang_dependence_1983}
Tsang Y~W, Witherspoon P~A (1983) The dependence of fracture mechanical and
  fluid flow properties on fracture roughness and sample size. Journal of
  Geophysical Research 88:2359, ISSN 0148-0227

\bibitem[Vogler et~al.(2017)]{vogler_comparison_2017}
Vogler D, Walsh S~D~C, Bayer P, Amann F (2017) Comparison of surface properties
  in natural and artificially generated fractures in a crystalline rock. Rock
  Mechanics and Rock Engineering 50(11):2891--2909, ISSN 0723-2632, 1434-453X

\bibitem[Xie et~al.(1999)]{xie_multifractal_1999}
Xie H, Wang J~A, Kwaśniewski M (1999) Multifractal characterization of rock
  fracture surfaces. International Journal of Rock Mechanics and Mining
  Sciences 36(1):19--27, ISSN 13651609

\bibitem[Yu et~al.(2020)]{yu_effects_2020}
Yu X, Regenauer-Lieb K, Tian F~B (2020) Effects of surface roughness and
  derivation of scaling laws on gas transport in coal using a fractal-based
  lattice boltzmann method. Fuel 259:116229, ISSN 00162361

\bibitem[Zambrano et~al.(2019)]{zambrano_analysis_2019}
Zambrano M, Pitts A~D, Salama A, Volatili T, Giorgioni M, Tondi E (2019)
  Analysis of fracture roughness control on permeability using {SfM} and fluid
  flow simulations: Implications for carbonate reservoir characterization.
  Geofluids 2019:4132386, ISSN 1468-8115, 1468-8123

\bibitem[Zanin et~al.(2012)]{ZaninEtAl2012}
Zanin M, Zunino L, Rosso O~A, Papo D (2012) Permutation entropy and its main
  biomedical and econophysics applications: A review. Entropy 14(8):1553--1577

\bibitem[Zhao et~al.(2014)]{zhao_effects_2014}
Zhao Z, Li B, Jiang Y (2014) Effects of fracture surface roughness on
  macroscopic fluid flow and solute transport in fracture networks. Rock
  Mechanics and Rock Engineering 47(6):2279--2286, ISSN 0723-2632, 1434-453X

\bibitem[Zimmerman and Bodvarsson(1996)]{zimmerman_hydraulic_1996}
Zimmerman R, Bodvarsson G (1996) Hydraulic conductivity of rock fractures.
  Transport in Porous Media 23(1):1--30

\end{thebibliography}

\appendix
\section*{Appendix: Proof of Theorem~\ref{thm:1}}

\normalsize
For simplicity, it is assumed throughout in the sequel (and without further mention) that $(X_t)_{t\in [0,\infty [}$ is a Gaussian stochastic process satisfying conditions {\bf\eqref{a1}--\eqref{a4}}. In some of the statements below, these assumptions could be weakened insignificantly, but this is not relevant for proving Theorem~\ref{thm:1}. Further assumptions will be mentioned explicitly where they are needed.

In the following, the \emph{lower and upper Hurst exponent} of $(X_t)_{t\in [0,\infty [}$ are defined by
\begin{align}\label{eq:HurstExp2}
 \lH&:=\inf\left\{s\in\R: \liminf_{\tau\to\infty} \frac{|c(\tau)|}{\tau^{2s-2}}=0\right\} \text{ and }\notag\\ \uH&:=\inf\left\{s\in\R: \limsup_{\tau\to\infty} \frac{|c(\tau)|}{\tau^{2s-2}}=0\right\},
\end{align}
respectively.
In general, the (upper) Hurst exponent provides an upper bound for the asymptotic behavior of the variance $\Var(X_\tau)$ and an assumption like the existence of the limit in \eqref{a5} implies that the Hurst exponent determines the asymptotic behavior of $\Var(X_\tau)$ completely. This is a consequence of the fact that $\Var(X_\tau)$ can be written in terms of the autocorrelation function $c$. For any $n\in\N$, one has
\begin{align}
\Var(X_{n}) &= \Cov\left(\sum_{i=0}^{n-1} \left(X_{i+1}-X_i\right),\sum_{j=0}^{n-1} \left(X_{j+1}-X_j\right)\right)\notag \\&= \sum_{i=0}^{n-1}\sum_{j=0}^{n-1} \Cov(X_{i+1}-X_{i},X_{j+1}-X_{j})\notag \\ \label{eq:var-ac-rel}
&= \Var(X_1)\left(n c(0)+2\sum_{k=1}^{n-1} (n-k)c(k)\right)\notag\\
&= \Var(X_1)\left(n +2\sum_{k=1}^{n-1} (n-k)c(k)\right).
\end{align}
Although the values of $\Var(X_\tau)$ are of interest for any $\tau \in \R$ with $\tau>0$, see e.g.\ \eqref{eq:Gaussian-formula}, the above relation \eqref{eq:var-ac-rel} for integer delays $n \in \N$ will be essential for proving Theorem\ref{thm:1}. 

The following observation will be useful in the sequel.
\begin{lemma}\label{lem:1}
Let $\beta\geq -1$, $\tau_0\in\N$ and $C>0$. Then there exists some $\tau_1>\tau_0$ such that, for any $n\geq \tau_1$,
\begin{equation}\label{eq:ineq-lem1}
C\cdot n < \sum_{k=\tau_0}^{n-1} (n-k) k^{\beta}.
\end{equation}
Moreover,
\begin{equation}\label{eq:ineq2-lem1}
\lim_{n\to\infty}\frac{\log\left(\sum_{k=\tau_0}^{n-1} (n-k)k^\beta\right)}{\log(n)}=\beta+2.
\end{equation}
\end{lemma}
\begin{proof}
To verify inequality \eqref{eq:ineq-lem1} for $n$ large enough, it is shown that $1/n$ times the right hand side tends to $+\infty$, as $n\to\infty$. For $\beta>-1$, it holds 
\begin{align*}
\frac 1n \sum_{k=\tau_0}^{n-1} (n-k) k^{\beta}
&=\sum_{k=\tau_0}^{n-1} k^{\beta}-\frac 1n \sum_{k=\tau_0}^{n-1}  k^{\beta+1}
\geq  \int_{\tau_0}^{n-1} k^{\beta} \diff k-\frac 1n \int_{\tau_0}^{n} k^{\beta+1} \diff k\\
&= \frac{(n-1)^{\beta+1}-\tau_0^{\beta+1}}{\beta+1}-\frac {n^{\beta+1}-n^{-1}\tau_0^{\beta+2}}{\beta+2}\\
&=\frac{(\beta+2)\left(\frac{n-1}n\right)^{\beta+1}-(\beta+1)}{(\beta+1)(\beta+2)} n^{\beta+1}-\frac{\tau_0^{\beta+1}}{\beta+1}+\frac 1n \frac{\tau_0^{\beta+2}}{\beta+2}. 
\end{align*}
The first term in this last expression tends to $+\infty$, as $n\to\infty$, while the other two summands stay bounded. This shows that the right hand side of \eqref{eq:ineq-lem1} grows faster than $n$ as $n\to\infty$ such that the inequality \eqref{eq:ineq-lem1} must be valid for $n$ large enough. For $\beta=-1$, it holds $\frac 1n \sum_{k=\tau_0}^{n-1} (n-k) k^{\beta}=\sum_{k=\tau_0}^{n-1} k^{-1}-1 +\frac {\tau_0}{n}\to+\infty$, as $n\to\infty$. Thus, inequality \eqref{eq:ineq-lem1} is also valid in this case for $n$ large enough. Recalling that the harmonic series grows like $\log(n)$, it is easily seen that $\lim_{n\to\infty}\log\left(\sum_{k=\tau_0}^{n-1} (n-k)k^\beta\right)/\log(n)=\lim_{n\to\infty}\log\left(n\log(n)\right)/{\log(n)}=1=\beta+2$ for $\beta=-1$ and thus \eqref{eq:ineq2-lem1} holds.

To see that \eqref{eq:ineq2-lem1} holds for $\beta>-1$ as well, multiply $n$ in the above estimate and observe that there is some constant $\tilde C>0$ such that, for all $n$ sufficiently large,
$$
\sum_{k=\tau_0}^{n-1} (n-k) k^{\beta}\geq \tilde C n^{\beta+2}.
$$
This clearly implies $\liminf_{n\to\infty}\log\left(\sum_{k=\tau_0}^{n-1} (n-k)k^\beta\right)/\log(n)\geq\beta+2$. For the opposite inequality observe that, for $\beta>-1$,
\begin{align*}
\frac 1n \sum_{k=\tau_0}^{n-1} (n-k) k^{\beta}
&\leq\sum_{k=\tau_0}^{n-1} k^{\beta}
\leq  \int_{\tau_0-1}^{n} k^{\beta} \diff k 
\leq (\beta+1)^{-1} n^{\beta+1},
\end{align*}
which implies $\limsup_{n\to\infty}\log\left(\sum_{k=\tau_0}^{n-1} (n-k)k^\beta\right)/\log(n)\leq\beta+2$. \qed
\end{proof}

A first consequence of Lemma~\ref{lem:1} is the following observation regarding the asymptotic behavior of the autocorrelation.
\begin{lemma}\label{lem:2}
Given $(X_t)_{t\in [0,\infty [}$, for any $s\geq\frac 12$ it holds
\begin{align*}
    \limsup_{\tau\to\infty}\frac{c(\tau)}{\tau^{2s-2}}\geq 0.
  \end{align*}
\end{lemma}
  \begin{proof}
     Assume that the left-hand-side of the inequality above is strictly less than $0$ for some $s\geq\frac 12$. Then there exists some constant $C>0$ and some $\tau_0\in \N$ such that $$
     c(\tau)\leq -C \tau^{2s-2}
     $$
     holds for all $\tau\geq \tau_0$. From \eqref{eq:var-ac-rel} it is  inferred that, for any $n\geq \tau_0$,
     \begin{align*}
       0&\leq\frac{\Var(X_{n})}{\Var(X_{1})}=n+2\sum_{k=1}^{n-1} (n-k)c(k)\\
       &=n+2\cdot\left(\sum_{k=1}^{\tau_0-1} (n-k)c(k)+\sum_{k=\tau_0}^{n-1} (n-k)c(k)\right)\\
       &\leq \left(1+2\sum_{k=1}^{\tau_0-1} \lvert c(k)\rvert\right)\cdot n - 2C \sum_{k=\tau_0}^{n-1} (n-k) k^{2s-2}.
     \end{align*}
     Noting that the term in parentheses is a constant, Lemma~\ref{lem:1} implies that the second summand in this expression dominates the first one, provided $n$ is large enough. Hence for large $n$ the last expression is negative, a contradiction. \qed
  \end{proof}

Under the stated assumptions, Lemma~\ref{lem:2} shows in particular that, if the limit in \eqref{a5} is assumed to exist (for some $H\geq 0.5$), then it cannot be negative. The next statement clarifies that assuming the existence of the limit \eqref{a5} for some $H\in\R$ implies that $H$ is the Hurst exponent of the process.
\begin{lemma} \label{lem:3}
Assume for $(X_t)_{t\in [0,\infty [}$ that the limit in \eqref{a5} exists and equals $C$ for some constants $H\in\R$ and $C\neq 0$. Then
for any $s>H$, $\lim_{\tau\to\infty} \frac{c(\tau)}{\tau^{2s-2}}=0$, and,
$\lim_{\tau\to\infty} \frac{c(\tau)}{\tau^{2s-2}}\in\{+\infty ,-\infty\}$, for any $s<H$, where the sign corresponds to the sign of $C$.
In particular, the Hurst exponent of $(X_t)_{t\in [0,\infty [}$ (as defined in \eqref{eq:HurstExp}) exists and equals $H$. Consequently, the constant $C$ is positive.
\end{lemma}
\begin{proof}
  Observe that, for any $s\in\R$,
  $$ \frac{c(\tau)}{\tau^{2s-2}}=\frac{c(\tau)}{\tau^{2H-2}}\cdot \tau^{2(H-s)}.$$
  Now the assumptions imply that the first factor on the right converges (to the constant $C\neq 0$), as $\tau\to\infty$, while the second factor tends to zero for $s>H$ and to $+\infty$ for $s<H$. From this the first assertion is obvious. It is an immediate consequence, that the number $H$ is the common value of the lower and upper Hurst exponent of $(X_t)_{t\in [0,\infty [}$ as defined in \eqref{eq:HurstExp2}, i.e. $H=\lH=\uH$. Moreover, the Hurst exponent in the sense of \eqref{eq:HurstExp} also exists and equals $H$. (The existence of the limit in \eqref{a5} implies, that for any $\delta$ such that $0<\delta<C$ there is some $\tau_0$ such that $(|C|-\delta)\cdot \tau^{2H-2}\leq |c(\tau)|\leq (|C|+\delta)\cdot \tau^{2H-2}$ for all $\tau\geq \tau_0$. Taking logarithms in this inequality and dividing by $\log(\tau)$, yields
  \begin{align*}
    2H-2+\frac{\log(|C|-\delta)}{\log(\tau)}\leq \frac{\log(|c(\tau)|)}{\log(\tau)}\leq 2H-2+\frac{\log(|C|+\delta)|)}{\log(\tau)}.
  \end{align*}
  Letting $\tau\to\infty$, the expressions on the left and on the right converge to $2H-2$ from which the existence of the limit in \eqref{eq:HurstExp} is clear.) \qed
\end{proof}

The next statement gives a general upper bound for the asymptotics of the variance $\Var(X_\tau)$ as $\tau\to\infty$ of a process $(X_t)_{t\in [0,\infty [}$ in terms of its upper Hurst exponent.

\begin{proposition} \label{prop:1} 
If $H$ is the upper Hurst exponent of $(X_t)_{t\in [0,\infty [}$ (as defined in \eqref{eq:HurstExp2}), then
    \begin{align*}
  \limsup_{n\to \infty} \frac{\log(\Var(X_{n}))}{\log(n)}\leq \max\{1,2H\}.
\end{align*}
\end{proposition}
\begin{proof}
  Observe that due to the relation \eqref{eq:var-ac-rel}, it is enough to show that
  \begin{equation}\label{eq:2H-upperbd}
\limsup_{n\to \infty}\frac{\log \left(n +2\sum_{k=1}^{n-1} (n-k)c(k)\right)}{\log (n)} \leq \max\{1,2H\}.
\end{equation}
Let $\eps>0$. The definition of the upper Hurst exponent \eqref{eq:HurstExp2} implies there exists some $\tau_0 \in \mathbb N$ such that
\begin{equation*}
\lvert c(\tau)\rvert \leq \tau^{2H-2+\epsilon}
\end{equation*}
holds for all $\tau \geq \tau_0$. It is inferred that, for $n>\tau_0$,
\begin{align*}
n +2\sum_{k=1}^{n-1} (n-k)c(k)&\leq n +2\sum_{k=1}^{\tau_0-1} (n-k)\lvert c(k) \rvert + 2\sum_{k=\tau_0}^{n-1} (n-k)\lvert c(k) \rvert\\
&\leq n\left(1 +2\sum_{k=1}^{\tau_0-1} \lvert c(k) \rvert\right) + 2\sum_{k=\tau_0}^{n-1} (n-k) k^{2H-2+\eps},
\end{align*}
where the expression in parentheses is a constant independent of $n$. Let us denote it by $E$.
Applying the formula $\log(a+b) \leq b/a + \log(a)$ (which is valid for $a,a+b>0$), with $a:=2\sum_{k=\tau_0}^{n-1} (n-k) k^{2H-2+\eps}$ and $b:=n\cdot E$ it is inferred that
\begin{align*}
&\limsup_{n\to \infty} \frac{\log \left(n +2\sum_{k=1}^{n-1} (n-k)c(k)\right)}{\log (n)}\\
&\quad\leq\limsup_{n\to \infty} \frac{\log \left(n\cdot E + 2\sum_{k=\tau_0}^{n-1} (n-k)k^{2H-2+\eps}\right)}{\log (n)}\\
&\quad\leq \limsup_{n\to \infty} \left(
\frac{n\cdot E}{\log (n)\cdot 2\sum_{k=\tau_0}^{n-1} (n-k)k^{2H-2+\eps}}\right.\\
&\qquad\qquad\qquad\qquad\left.+ \frac{\log \left(2\sum_{k=\tau_0}^{n-1}(n-k)k^{2H-2+\eps}\right)}{\log (n)}
\right).
\end{align*}
Now observe that for $H\geq 0.5$, by the inequality 
in Lemma~\ref{lem:1}, the first quotient vanishes as $n\to\infty$, and by \eqref{eq:ineq2-lem1} in Lemma~\ref{lem:1}, the second quotient converges to $2H+\eps$.
Since this estimates holds for any $\eps>0$, the assertion follows for any $H\geq 0.5$ by letting $\eps\searrow 0$.

For $H<0.5$, the summand $a$ from above can be bounded as follows
\begin{align*} 
a
\leq 2n^{1+\eps}\sum_{k=\tau_0}^{n-1} k^{2H-2}
\leq2{n^{1+\eps}}\int_{\tau_0-1}^{n} k^{2H-2} \diff k
\leq {n^{1+\eps}}\frac{2(\tau_0-1)^{2H-1}}{1-2H}. 
\end{align*}
Since this holds for any $\eps>0$ (and since the summand $b$ grows like $n$, as $n\to\infty$), one concludes that $1$ is an upper bound for the limes superior in \eqref{eq:2H-upperbd}, as asserted for $H<0.5$. This completes the proof. \qed 
\end{proof}

\begin{proposition} \label{prop:2} Let $H\geq 0.5$ and $C>0$. 
    Assume that for $(X_t)_{t\in [0,\infty [}$ the limit in \eqref{a5} exists and equals $C$. Then $H$ is the Hurst exponent of $(X_t)_{t\in [0,\infty [}$ (as defined in \eqref{eq:HurstExp}) and
    \begin{align*}
  \lim_{n\to \infty} \frac{\log(\Var(X_{n}))}{\log(n)}= 2H.
\end{align*}
\end{proposition}
\begin{proof} By Lemma~\ref{lem:3}, $H$ is indeed the Hurst exponent of $(X_t)_{t\in [0,\infty [}$.
In view of Proposition~\ref{prop:1} and the relation \eqref{eq:var-ac-rel}, the only thing that remains to be shown is that
\begin{equation*}
L:=\liminf_{n\to \infty}\frac{\log\left(n +2\sum_{k=1}^{n-1} (n-k)c(k)\right)}{\log(n)} \geq 2H.
\end{equation*}
Since the limit in \eqref{a5} is assumed to equal $C>0$, 
there exists $\tau_0 \in \mathbb N$ such that
\begin{equation*}
\frac C2\tau^{2H-2} \leq c(\tau) 
\end{equation*}
holds for all $\tau \geq \tau_0$. In particular, $c(\tau)$ will be positive for $\tau>\tau_0$.
This relation implies that, for any $n\in\N$,
\begin{align*}
n +2\sum_{k=1}^{n-1} (n-k)c(k)&=n +2\sum_{k=1}^{\tau_0-1} (n-k)c(k) +2\sum_{k=\tau_0}^{n-1} (n-k)c(k)\\ 
&\geq
-\left \lvert n +2\sum_{k=1}^{\tau_0-1} (n-k)c(k) \right \rvert+ C\sum_{k=\tau_0}^{n-1} (n-k)k^{2H-2}\\ 
&\geq
-n\cdot \left(1 +2\sum_{k=1}^{\tau_0-1} \lvert c(k) \rvert\right)+ C\sum_{k=\tau_0}^{n-1} (n-k)k^{2H-2}.
\end{align*}
Now Lemma~\ref{lem:1} ensures that there exists a $\tau_1$ such that the last expression is positive for all $n\geq \tau_1$.
Employing the formula $\log(a+b) \geq b/a -(b/a)^2/2+ \log(a)$, which is valid for $a,a+b>0$, allows to conclude that
\begin{align*}
L&\geq\liminf_{n\to \infty} \frac{\log \left(C\sum_{k=\tau_0}^{n-1} (n-k)k^{2H-2}-E n\right)}{\log (n)}\\
&\geq \liminf_{n\to \infty} \frac 1{\log(n)}\Biggl(
\frac{-E \cdot n}{C\sum_{k=\tau_0}^{n-1} (n-k)  k^{2H-2}}
- \frac{E^2 \cdot n^2}{2\left(C\sum_{k=\tau_0}^{n-1} (n-k)  k^{2H-2}\right)^2}\\
&\hspace{3cm}+ \log \left(C\sum_{k=\tau_0}^{n-1} (n-k)\cdot k^{2H-2}\right)\Biggr). 
\end{align*}
By the inequality in Lemma~\ref{lem:1}, the first two summands are bounded and therefore, when divided by $\log(n)$, they vanish as $n\to\infty$. Furthermore, by \eqref{eq:ineq2-lem1}, the third summand divided by $\log(n)$ converges to $2H$.  Notice that for the application of Lemma~\ref{lem:1} the assumption $H\geq 0.5$ was used. \qed
\end{proof}

After these preparations a proof of Theorem~\ref{thm:1} can be provided.
\begin{proof}[of Theorem~\ref{thm:1}] 
Under the assumptions of the theorem, Proposition~\ref{prop:2} applies. Thus, $\lim_{n \to \infty}\frac{\log_2(\Var(X_{n}))}{\log_2(n)}$ exists and equals $2H$. One can compute this limit by passing to any subsequence $(\tau_n)_{n\in\N} \subseteq \N$ diverging to $\infty$. Choose $\tau_n=2^n$. Then
\begin{align*}
H &= \frac{1}{2}\lim_{n\to \infty} \frac{\log_2(\Var(X_{2^n}))}{\log_2(2^n)}\\
&= \frac{1}{2}\lim_{n\to \infty} \frac{1}{n} \log_2(\Var(X_{2^{n}}))\notag\\
&= \frac{1}{2}\lim_{n\to \infty} \frac{1}{n} \left[ \log_2(\Var(X_{2^{n}})) -\log_2(\Var(X_{1})) \right]\notag\\
&= \frac{1}{2}\lim_{n\to \infty} \frac{1}{n}\sum_{i=0}^{n-1} \log_2\left(\frac{\Var(X_{2^{i+1}})}{\Var(X_{2^{i}})}\right).\notag\\
\end{align*}
Together with Proposition~\ref{propchar} and the Stolz-Ces\'aro theorem, this implies
\begin{align*}
&\liminf_{\tau \to \infty} h(p(\tau)) \leq \liminf_{n\to \infty} h(p(2^n)) = \frac{1}{2}\liminf_{n\to \infty} \log_2\left(\frac{\Var(X_{2^{n+1}})}{\Var(X_{2^{n}})}\right)\\
&\leq \frac{1}{2}\liminf_{n\to \infty} \frac{1}{n}\sum_{i=0}^{n-1} \log_2\left(\frac{\Var(X_{2^{i+1}})}{\Var(X_{2^{i}})}\right) = H\\
& = \frac{1}{2}\limsup_{n\to \infty} \frac{1}{n}\sum_{i=0}^{n-1} \log_2\left(\frac{\Var(X_{2^{i+1}})}{\Var(X_{2^{i}})}\right)\\
& \leq \frac{1}{2}\limsup_{n\to \infty} \log_2\left(\frac{\Var(X_{2^{n+1}})}{\Var(X_{2^{n}})}\right)\\
&= \limsup_{n\to \infty} h(p(2^n)) \leq \limsup_{\tau \to \infty} h(p(\tau)).
\end{align*}
So in particular, if the limit $\lim_{\tau \to \infty} h(p(\tau))$ exists, it must be equal to $H$, which completes the proof of the theorem.\qed
%
\end{proof}

Recall that it was assumed in Theorem~\ref{thm:1} that the limit $\lim_{\tau \to \infty} h(p(\tau))$ exists in order to conclude that it must be equal to $H$.  The last sequence of inequalities in the above proof shows that, alternatively, one can impose the weaker assumption that $\lim_{n \to \infty} h(p(2^n))$ exists in order to prove the (slightly weaker) statement  that $\lim_{n \to \infty} h(p(2^n))=H$ holds.

\end{document}